\documentclass[]{article}

\usepackage{graphicx}
\usepackage{amsmath}
\usepackage{amssymb}
\usepackage{amsthm}
\usepackage{pxfonts}
\usepackage{enumerate}
\usepackage{color}
\usepackage{mathdots}
\usepackage{sectsty}
\usepackage{tikz}
\usepackage{adjustbox}
\usepackage{caption}
\usepackage{bbold}
\usepackage[hidelinks]{hyperref}
\allowdisplaybreaks

\sectionfont{\scshape\centering\fontsize{11}{14}\selectfont}
\subsectionfont{\scshape\fontsize{11}{14}\selectfont}
\usepackage{fancyhdr}

\newcommand\shorttitle{Random entire functions
from random polynomials with real zeros}
\newcommand\authors{Theodoros Assiotis}

\newtheorem{thm}{Theorem}[section]
\newtheorem{cor}[thm]{Corollary}
\newtheorem{lem}[thm]{Lemma}
\newtheorem{defn}[thm]{Definition}
\newtheorem{rmk}[thm]{Remark}
\newtheorem{prop}[thm]{Proposition}
\newtheorem*{theorem*}{Theorem}

\title{\large \bf RANDOM ENTIRE FUNCTIONS FROM RANDOM POLYNOMIALS WITH REAL ZEROS}
\author{\small THEODOROS ASSIOTIS}
\date{}

\begin{document}

\maketitle

\begin{abstract}
We point out a simple criterion for convergence of polynomials to a concrete entire function in the Laguerre-P\'{o}lya ($\mathcal{LP}$) class (of all functions arising as uniform limits of polynomials with only real roots). We then use this to show that any random $\mathcal{LP}$ function can be obtained as the uniform limit of rescaled characteristic polynomials of principal submatrices of an infinite unitarily invariant random Hermitian matrix. Conversely, the rescaled characteristic polynomials of principal submatrices of any infinite random unitarily invariant Hermitian matrix converge uniformly to a random $\mathcal{LP}$ function. This result also has a natural extension to $\beta$-ensembles. Distinguished cases include random entire functions associated to the $\beta$-Sine, and more generally $\beta$-Hua-Pickrell, $\beta$-Bessel and $\beta$-Airy point processes studied in the literature. 
\end{abstract}

\section{Introduction}

The problem of understanding the scaling limit of eigenvalues of random matrices to some limiting random point process is one of the most fundamental in random matrix theory and has been studied for many decades \cite{ForresterBook}. On the other hand, the very natural (and more general) problem of understanding the scaling limit of the characteristic polynomial itself to some random entire function with appropriate point process of zeros has only seen progress during the last decade. This problem is also partly motivated by the connection between number theory and random matrices \cite{Montgomery,KeatingSnaith}, as the characteristic polynomial of random unitary matrices can be considered a good model for the Riemann zeta function, see the introduction of \cite{ChhaibiNajnudelNikeghbali} for more details.

As far as we are aware the first result on this problem is contained in \cite{ChhaibiNajnudelNikeghbali} where the scaling limit of the characteristic polynomial of random unitary matrices was established and a limiting entire function whose zeros are given by the determinantal point process with the sine kernel was constructed and its properties studied. The authors make use of the determinantal point process structure and of some earlier quantitative estimates from \cite{MaplesNajnudelNikeghbali} to establish their main result. This approach was extended in \cite{ChhaibiHovhannisyanEtAl} to a class of point processes called product amenable for which the limiting entire functions enjoy a principal value product representation. Moreover, convergence of characteristic polynomials of certain Wigner matrices to the entire function constructed in \cite{ChhaibiNajnudelNikeghbali} is a consequence of the results of \cite{AizenmanWarzel}, see \cite{Sodin} where this statement is made explicit. More recently in \cite{ValkoViragZeta} the authors considered the scaling limit of the circular $\beta$-ensemble (random unitary matrices correspond to $\beta=2$) characteristic polynomial and various properties of the limiting random entire function, which they called the stochastic zeta function, were studied. Their approach is based on viewing the random characteristic polynomial as a Fredholm determinant of an appropriate stochastic operator. This approach was then also followed in \cite{ValkoLin} to study two families of random entire functions arising from random matrices which we call here the stochastic Hua-Pickrell (this generalises the stochastic zeta function) and stochastic Bessel functions. Finally, the scaling limit of the characteristic polynomial of Gaussian matrices at the soft edge was studied in \cite{StochasticAiryFunction} and the so-called stochastic Airy function was constructed and studied. The authors there begin by producing a recurrence relation for the characteristic polynomials which they then go on to study as a kind of random dynamical system.

In this paper we follow a different approach from previous works, taking a viewpoint motivated by \cite{OlshanskiVershik,BorodinOlshanskiErgodic,OrbitalBeta}, to study scaling limits of characteristic polynomials of random matrices. We begin by pointing out a little framework, based on a rather simple criterion and variations of it, for proving convergence of a sequence of polynomials to a concrete entire function in  the Laguerre-P\'{o}lya class. This is the class of all entire functions which arise as uniform on compact sets limits of polynomials with only real roots. Considering this class is rather natural from the perspective of taking limits of random matrix characteristic polynomials and all random entire functions mentioned above belong to it almost surely. Some of the complex analysis results proven in the sequel seem to us that they should be classical but we have not been able to locate the exact statements in the literature (see the discussion before Proposition \ref{ConvergenceProp}). On the other hand, the fact that combined with previous probabilistic work they have highly non-trivial consequences (which, as far as we can tell, do not follow by a different method) for characteristic polynomials of random matrices, see Theorem \ref{HermitianTheorem} and Theorem \ref{InterlacingTheorem}, is novel and is the main message of this paper. Theorem \ref{HermitianTheorem} can informally be stated as follows, see the next sections for the required background:

\begin{theorem*}
Let $\mathbf{H}$ be a random infinite Hermitian matrix whose $N\times N$ principal submatrices $\mathbf{H}_N$ are unitarily invariant in law for all $N\ge 1$. Then, almost surely with respect to the law of $\mathbf{H}$, the sequence of rescaled (reverse) characteristic polynomials
\begin{equation*}
    \Psi_N\left(\frac{z}{N}\right)=\det\left(\mathbf{I}-\frac{z}{N}\mathbf{H}_N\right)
\end{equation*}
converges uniformly on compact sets in $\mathbb{C}$ to a (possibly) random Laguerre-P\'{o}lya function. Moreover, any random Laguerre-P\'{o}lya function, normalised to be $1$ at $0$, can be realised in this way.
\end{theorem*}

We note that the papers \cite{ValkoViragZeta,ValkoLin,StochasticAiryFunction} also provide quantitative convergence rates to the limiting entire functions studied there. Furthermore, many remarkable properties of these functions have been studied: representations as principal value products \cite{ChhaibiNajnudelNikeghbali, ChhaibiHovhannisyanEtAl,ValkoViragZeta}, equivalent descriptions in terms of stochastic equations \cite{StochasticAiryFunction}, Taylor coefficients given in terms of iterated stochastic integrals \cite{ValkoViragZeta,ValkoLin}, which are related to some classical identities for Brownian motion \cite{Dufresne,RiderValko}, and explicit moment formulae \cite{ValkoViragZeta,ValkoLin}. For generic random Laguerre-P\'{o}lya functions that we consider in this paper such precise results are highly unlikely to exist and in principle the most one could hope for is a convergence statement of the sort we prove here. In fact, even in the special case of the stochastic Airy function \cite{StochasticAiryFunction} whose study is in some sense  (both technically and conceptually) the most challenging out of the entire functions arising from classical random matrix theory point processes, explicit moment formulae have not been discovered yet.

It is then fitting to conclude the introduction with a question in the direction of explicit formulae and integrability. In the papers \cite{DistinguishedFamily,CircularJacobiJoint,ForresterJoint}, motivated by different considerations, the main object of study is the distribution of the first (non-trivial) Taylor coefficient of the stochastic Hua-Pickrell and stochastic Bessel functions. These distributions have connections to Painlev\'{e} equations and for certain parameters admit explicit combinatorial formulae for their moments. It would be very interesting if analogous results exist for the stochastic Airy function of \cite{StochasticAiryFunction}. More generally, do such results exist for higher order Taylor coefficients for any of these functions?

A couple of days after this paper was first posted on $\mathsf{arXiv}$ the very interesting paper \cite{NajnudelNikeghbali} appeared as well. Although both papers are concerned with convergence to random entire functions with real zeros the points of view are very different and there is essentially no overlap in terms of results and techniques between the two. The authors of both papers were working completely independently.

The paper is organised as follows. In Section \ref{SectionConv} we discuss convergence to Laguerre-P\'{o}lya functions. Section \ref{SectionHermitian} contains our results on unitarily invariant Hermitian matrices. The extension to $\beta$-ensembles is presented in Section \ref{SectionBeta}. Proofs of the complex analysis results are contained in Section \ref{SectionProofs}.

\paragraph{Acknowledgements} I am grateful to Alexei Borodin for useful comments on a preliminary version of the paper. I am also very grateful to an anonymous referee for a careful reading of the paper and useful comments and suggestions.

\section{Convergence to Laguerre-P\'{o}lya entire functions}\label{SectionConv}
We begin by defining the following infinite-dimensional parameter space.
\begin{defn}
Define the space $\hat{\Omega}$:
\begin{align*}
\hat{\Omega}&= \bigg\{\omega=\left(\alpha^+,\alpha^-,\gamma_1,\delta\right)\in \mathbb{R}_+^\infty\times \mathbb{R}_+^\infty \times \mathbb{R}\times \mathbb{R}_+:\\
&\alpha^+=(\alpha_1^+\ge \alpha_2^+\ge \cdots \ge 0);\alpha^-=(\alpha_1^-\ge \alpha_2^-\ge \cdots \ge 0);\sum_{i=1}^\infty\left(\alpha_i^+\right)^2 +\sum_{i=1}^\infty\left(\alpha_i^-\right)^2\le \delta \bigg\}.
\end{align*}
Endow $\hat{\Omega}$ with the topology of coordinate-wise convergence. 
$\hat{\Omega}$ is in bijection with the space $\Omega$:
\begin{align*}
\Omega&=\bigg\{\omega=\left(\alpha^+,\alpha^-,\gamma_1,\gamma_2\right)\in \mathbb{R}_+^\infty\times \mathbb{R}_+^\infty \times \mathbb{R}\times \mathbb{R}_+: \\
&\alpha^+=(\alpha_1^+\ge \alpha_2^+\ge \cdots \ge 0);\alpha^-=(\alpha_1^-\ge \alpha_2^-\ge \cdots \ge 0);\sum_{i=1}^\infty\left(\alpha_i^+\right)^2 +\sum_{i=1}^\infty\left(\alpha_i^-\right)^2<\infty\bigg\}.
\end{align*}
via the correspondence $\gamma_2=\delta- \sum_{i=1}^\infty\left(\alpha_i^+\right)^2 -\sum_{i=1}^\infty\left(\alpha_i^-\right)^2$. We endow $\Omega$ with the topology making this bijection bi-continuous.
\end{defn}

In principle we could have only defined the space $\hat{\Omega}$ but as the parameter $\gamma_2$ makes formulae in the sequel look nicer we introduce $\Omega$ as well. We move on to the definition of the Laguerre-P\'{o}lya ($\mathcal{LP}$) class.

\begin{defn}\label{DefLP} We define the Laguerre-P\'{o}lya ($\mathcal{LP}$) class of entire functions, parametrised by $\omega\in \Omega$ (or equivalently $\hat{\Omega}$), consisting of functions $\mathsf{E}_\omega$ given by the Hadamard product:
\begin{align}
  \mathsf{E}_\omega(z)=e^{-\gamma_1z-\frac{\gamma_2}{2}z^2}\prod_{i=1}^\infty e^{z\alpha_i^+}\left(1-z\alpha_i^+\right)\prod_{i=1}^\infty e^{-z \alpha_i^-}\left(1+z\alpha_i^-\right). 
\end{align}
 We endow $\mathcal{LP}$ with the topology of uniform convergence on compact sets in $\mathbb{C}$.
\end{defn}

This is the class of functions $f$ arising as uniform on compact sets limits of polynomials with only real zeros, subject to the constraint $f(0)=1$, see \cite{Widder,TotalPositivity}. This fact is the original result of Laguerre \cite{Laguerre} (who considered the case of only positive zeros) and P\'{o}lya \cite{PolyaOriginal} (who considered the general case). The standard parametrisation of Laguerre-P\'{o}lya functions in the literature is slightly different (and normally the constraint $f(0)=1$ is not imposed, allowing in particular for $z=0$ to correspond to a zero). It is given in terms of the roots $\{c_i\}$ of the function instead of their reciprocals and there is no positive and negative splitting. We chose this convention (which is equivalent, subject to $f(0)=1$ of course) to be consistent with previous random matrix theory works. These functions are also very closely related to the subject of total positivity \cite{TotalPositivity}. In particular, functions in $\mathcal{LP}$ (except $e^{-\gamma_1 z}$) are exactly the ones appearing as reciprocals of Laplace transforms of P\'{o}lya frequency functions, see \cite{TotalPositivity}. Moreover, they (more precisely the subclass considered by Laguerre) give a transcendental characterisation of so-called multiplier sequences, see \cite{PolyaSchur} and \cite{BorceaBranden} for a vast generalisation.

We also have a nice combinatorial power series expansion of $\mathsf{E}_{\omega}$. Consider the modified power sums associated to $\omega\in \Omega$ (or equivalently $\hat{\Omega}$):
\begin{align*}
 \tilde{p}_1(\omega)&=\gamma_1,\ \ \ 
 \tilde{p}_2(\omega)=\gamma_2+\sum_{i=1}^\infty \left(\alpha_i^+\right)^2+\sum_{i=1}^\infty \left(\alpha_i^-\right)^2=\delta,\\
 \tilde{p}_k(\omega)&=\sum_{i=1}^\infty\left(\alpha_i^+\right)^k+\sum_{i=1}^\infty(-1)^k\left(\alpha_i^-\right)^k, \ \ k\ge 3.
\end{align*}
Then, by computing the derivatives of the entire function $\mathsf{E}_\omega$, or alternatively using the Newton identities for polynomials and making use of the uniform limit in Proposition \ref{ConvergenceProp}, we get the power series expansion:
\begin{align}
\mathsf{E}_\omega(z)=1+\sum_{j=1}^\infty z^j \sum_{\substack{m_1+2m_2+\cdots+jm_j=j\\m_1\ge 0, \dots, m_j\ge 0}} \prod_{i=1}^j\frac{\left(-\tilde{p}_i(\omega)\right)^{m_i}}{m_i!i^{m_i}}.
\end{align}
Under some natural assumptions on the parameters $\gamma_1$ and $\gamma_2$ (which are known to hold for some of the random matrix models we consider later on), the function $\mathsf{E}_\omega$ also has a principal value product representation, see Proposition \ref{PrincipalValProp}.

Moving on, for any $N\ge 1$ define the Weyl chamber:
\begin{align}\label{WeylChamber}
\mathbb{W}_N=\{\mathbf{x}=(x_1,x_2,\dots,x_N)\in \mathbb{R}^N:x_1\ge x_2 \ge \cdots \ge x_N\}.
\end{align}
We consider the following embedding of $\mathbb{W}_N$ in $\hat{\Omega}$ (equivalently $\Omega$).
\begin{defn}
For $\mathbf{x}^{(N)}\in \mathbb{W}_N$ define the quantities:
\begin{align*}
  \alpha_i^{+,(N)}\left(\mathbf{x}^{(N)}\right)&=\begin{cases} \max\left\{x_i^{(N)},0 \right\}, &i=1,\dots, N, \\
  0, &i=N+1, N+2, \dots,
  \end{cases}\\
    \alpha_i^{-,(N)}\left(\mathbf{x}^{(N)}\right)&=\begin{cases} \max\left\{-x_{N-i+1}^{(N)},0 \right\},&i=1,\dots, N, \\ 
    0,&i=N+1, N+2, \dots,
  \end{cases}\\
  \gamma^{(N)}_1\left(\mathbf{x}^{(N)}\right)&=\sum_{i=1}^\infty \alpha_i^{+,(N)}\left(\mathbf{x}^{(N)}\right)-\sum_{i=1}^\infty\alpha_i^{-,(N)}\left(\mathbf{x}^{(N)}\right)=\sum_{i=1}^N x_i^{(N)},\\
  \delta^{(N)}\left(\mathbf{x}^{(N)}\right)&=\sum_{i=1}^\infty\left(\alpha_i^{+,(N)}\left(\mathbf{x}^{(N)}\right)\right)^2+\sum_{i=1}^\infty\left(\alpha_i^{-,(N)}\left(\mathbf{x}^{(N)}\right)\right)^2=\sum_{i=1}^N \left(x_i^{(N)}\right)^2.
\end{align*}
\end{defn}

Consider the following conditions, named after Olshanski-Vershik (O-V) who introduced them in \cite{OlshanskiVershik} (usually in the literature \cite{BorodinOlshanskiErgodic,OrbitalBeta} and in the original paper \cite{OlshanskiVershik} these definitions are given in terms of $\mathbf{x}^{(N)}/N$). 

\begin{defn}
We say that a sequence $\left\{\mathbf{x}^{(N)} \right\}_{N=1}^\infty$ in $\left\{\mathbb{W}_N \right\}_{N=1}^\infty$ satisfies the Olshanski-Vershik (O-V) conditions iff the following limits exist:
\begin{align*}
 \alpha_i^{\pm}&\overset{\textnormal{def}}{=}  \lim_{N\to \infty} \alpha_i^{\pm,(N)}\left(\mathbf{x}^{(N)}\right), \ \ \forall i\ge 1,\\
  \gamma_1&\overset{\textnormal{def}}{=}  \lim_{N\to \infty} \gamma_1^{(N)}\left(\mathbf{x}^{(N)}\right), \\
   \delta&\overset{\textnormal{def}}{=}  \lim_{N\to \infty} \delta^{(N)}\left(\mathbf{x}^{(N)}\right).
\end{align*}
Observe that in this case by Fatou's lemma we have:
\begin{align}\label{Gamma2}
 \gamma_2=\delta- \sum_{i=1}^\infty\left(\alpha_i^+\right)^2 -\sum_{i=1}^\infty\left(\alpha_i^-\right)^2  \ge 0.
\end{align}
We say that the point $\omega=\left(\alpha^+,\alpha^-,\gamma_1,\gamma_2\right)\in \Omega$  is the limit point of $\left\{\mathbf{x}^{(N)} \right\}_{N=1}^\infty$.
\end{defn}

We have the following simple criterion for convergence of a sequence of polynomials to a concrete entire function in $\mathcal{LP}$. Although, as far as we can tell, the exact statement has not appeared before in the literature the criterion should still be considered in some sense classical. It can be obtained by combining some classical results of Lindwart and P\'{o}lya \cite{LindwartPolya} along with the equivalent form of the (O-V) conditions in Proposition 2.3 of \cite{OrbitalBeta} and a little extra argument. However, a direct proof is actually short and instructive (and generalises in a way that we will need later on) and we present that instead in Section \ref{SectionProofs}. Unsurprisingly, the simple proof is classical complex analysis and the ideas go (at least) back to various proofs of the original Laguerre-P\'{o}lya result \cite{Widder,TotalPositivity}. 

\begin{prop}\label{ConvergenceProp}
A sequence $\left\{\mathbf{x}^{(N)} \right\}_{N=1}^\infty$ in $\left\{\mathbb{W}_N \right\}_{N=1}^\infty$ satisfies the (O-V) conditions, with limit point $\omega$, if and only if the following convergence holds uniformly on compact sets in $\mathbb{C}$:
\begin{equation}\label{ConvergenceDisplay}
 \Psi_N(z) \overset{\textnormal{def}}{=}\prod_{i=1}^N\left(1-zx_i^{(N)}\right)  \overset{N\to\infty}{\longrightarrow} \mathsf{E}_\omega(z).
\end{equation}
\end{prop}

In plain words, the result says that if we have convergence of the points (reciprocals of the roots) plus a little more (which is important), namely convergence of the sum of points and sum of squares of points\footnote{If these do not converge we can renormalise the polynomial as follows:
\begin{equation*}
  \prod_{i=1}^N\left(1-zx_i^{(N)}\right)e^{-\mathsf{C}_1^{(N)}z-\frac{1}{2}\mathsf{C}_2^{(N)}z^2},
\end{equation*}
for some real $\mathsf{C}_1^{(N)}, \mathsf{C}_2^{(N)}$ so that $\left(\alpha^{+,(N)}\left(\mathbf{x}^{(N)}\right),\alpha^{-,(N)}\left(\mathbf{x}^{(N)}\right),\gamma_1^{(N)}\left(\mathbf{x}^{(N)}\right)+\mathsf{C}_1^{(N)},\delta^{(N)}\left(\mathbf{x}^{(N)}\right)+\mathsf{C}_2^{(N)}\right)\in \hat{\Omega}$ for all $N$. If $\gamma_1^{(N)}\left(\mathbf{x}^{(N)}\right)+\mathsf{C}_1^{(N)}$ and $\delta^{(N)}\left(\mathbf{x}^{(N)}\right)+\mathsf{C}_2^{(N)}$ converge (of course we already assume that the $\alpha^{\pm,(N)}$ parameters converge), then we get uniform convergence on compact sets in $\mathbb{C}$ for the renormalised polynomial from the slight extension of Proposition \ref{ConvergenceProp} stated in Remark \ref{RemarkExtension}.} then the polynomials converge uniformly on compact sets in $\mathbb{C}$ to the corresponding entire function $\mathsf{E}_\omega$ in $\mathcal{LP}$. The converse is intuitive and easier to see. There is also a quantitative version, see Proposition \ref{PropositionQuantitative}. 

Our interest is in probabilistic applications of this result: suppose a random sequence $\left\{\mathbf{x}^{(N)} \right\}_{N=1}^\infty$ in $\left\{\mathbb{W}_N \right\}_{N=1}^\infty$ has law $\mathfrak{M}$ and $\mathfrak{M}$-a.s. satisfies the (O-V) conditions with limit point $\omega$ having law $\nu$. Then, $\mathfrak{M}$-a.s. the convergence (\ref{ConvergenceDisplay}) holds uniformly on compact sets in $\mathbb{C}$ where $\omega$ has law $\nu$. Intuitively a random function in $\mathcal{LP}$ is equivalent to a probability measure $\nu$ on $\Omega$, but this needs a proof (which builds upon an extension of the criterion above), see Proposition \ref{MeasureProp}. From now on, when we speak of a random $\mathcal{LP}$ function we then mean a probability measure on $\Omega$.

Our goal in this paper is not to prove the (O-V) conditions for new models but rather to point out that combining a rather simple complex analysis argument with known (non-trivial) probabilistic results has immediate and rather remarkable consequences for convergence to the corresponding random entire functions. We also observe that the question of universality of such random entire functions boils down to proving universality for the points (equivalently reciprocals of their roots), which is known, along with universality for the sum and sum of squares of points\footnote{And if needed we could renormalise the polynomial.}, which as far as we can tell is only known \cite{UniversalityStochasticBessel} at the hard edge of random matrices (under a subsequence)  but should be true more generally and would be interesting to investigate further. We discuss this briefly in Remark \ref{Universality}.

Finally, such convergence conditions (or variants thereof) exist for many probabilistic models from integrable probability and asymptotic representation theory, see for example \cite{VershikKerovSymmetric,VershikKerovUnitary, OlshanskiHarmonic, BorodinOlshanskiHarmonic,KerovOkounkovOlshanski}. 
 Proposition \ref{ConvergenceProp} then applies and has analogous consequences to the ones presented in the next sections. Since in this paper our focus is on random matrices we do not discuss these discrete models.
 
We conclude this section with a couple of comments on Proposition \ref{ConvergenceProp}.

\begin{rmk}\label{RemarkExtension}
The exact same proof of Proposition \ref{ConvergenceProp}, with only notational modifications, gives a more general result. Proposition \ref{ConvergenceProp} being a special case, can be seen if we observe:
\begin{align*}
 \prod_{i=1}^N\left(1-zx_i^{(N)}\right)&= e^{-\gamma_1^{(N)}\left(\mathbf{x}^{(N)}\right)z-\frac{1}{2}\left[\delta^{(N)}\left(\mathbf{x}^{(N)}\right)-\sum_{i=1}^\infty\left(\alpha_i^{+,(N)}\left(\mathbf{x}^{(N)}\right)\right)^2-\sum_{i=1}^\infty\left(\alpha_i^{-,(N)}\left(\mathbf{x}^{(N)}\right)\right)^2\right]z^2}\times \\ 
 &\times \prod_{i=1}^\infty e^{z\alpha_i^{+,(N)}\left(\mathbf{x}^{(N)}\right)}\left(1-z\alpha_i^{+,(N)}\left(\mathbf{x}^{(N)}\right)\right)\prod_{i=1}^\infty e^{-z \alpha_i^{-,(N)}\left(\mathbf{x}^{(N)}\right)}\left(1+z\alpha_i^{-,(N)}\left(\mathbf{x}^{(N)}\right)\right)  .
\end{align*}
Namely, suppose $\hat{\Omega} \ni \omega_N \longrightarrow \omega \in \hat{\Omega}$ (in the topology of $\hat{\Omega}$ of coordinate-wise convergence). Then, the following convergence holds uniformly on compact sets in $\mathbb{C}$: 
\begin{equation}
\mathsf{E}_{\omega_N}(z)\overset{N\to\infty}{\longrightarrow} \mathsf{E}_\omega(z).    
\end{equation}
In words, we have continuity in $\mathcal{LP}$ in the parameter space $\hat{\Omega}$. This again seems classical but we have not found the statement in the literature. A quantitative version also exists, see Proposition \ref{PropositionQuantitative}. This result will be relevant when we discuss the soft edge scaling of characteristic polynomials and entire functions with $\beta$-Airy point process zeros. Moreover, such a result, combined with the probabilistic work of \cite{HardToSoft}, could have consequences for the so-called hard-to-soft transition for these entire functions. 
\end{rmk}

\begin{rmk}\label{RemarkPositive}
If $\alpha_i^{-,(N)}\left(\mathbf{x}^{(N)}\right)\equiv 0$ we fall in the subclass $\mathcal{LP}_+$ of $\mathcal{LP}$ consisting of entire functions which are uniform limits of polynomials with only positive zeros (and subject to $f(0)=1$), see \cite{TotalPositivity}, in which case, see \cite{TotalPositivity}, we have $\sum_{i=1}^\infty \alpha_i^+ \le \gamma_1 <\infty$ and moreover $\gamma_2=0$. Hence, the convergence statement (\ref{ConvergenceDisplay}) takes the nicer form:
\begin{equation}\label{ConvergencePositive}
  \Psi_N(z) \overset{N\to\infty}{\longrightarrow}    e^{\left(-\gamma_1+\sum_{i=1}^\infty \alpha_i^{+}\right)z}\prod_{i=1}^\infty \left(1-z\alpha_i^+\right).  
\end{equation}
\end{rmk}

\paragraph{Notation} For a topological space $\mathfrak{X}$ we write $\mathcal{M}_p\left(\mathfrak{X}\right)$ for the space of Borel probability measures on $\mathfrak{X}$ and write $\mathsf{Law}\left(\mathsf{X}\right)$ for the law of a random element $\mathsf{X}$.

\section{Applications to unitarily invariant Hermitian matrices}\label{SectionHermitian}

\subsection{A general convergence theorem for random Hermitian matrices}

Consider the infinite-dimensional unitary group $\mathbb{U}(\infty)$, namely the inductive limit of the chain of $N\times N$ unitary groups $\mathbb{U}(N)$ under the natural inclusions. Denote by $\mathbb{H}(N)$ and by $\mathbb{H}$ the spaces of $N\times N$ and infinite Hermitian matrices respectively. We note that $\mathbb{H}$ can also be realised as the projective limit $\underset{\leftarrow}{\lim}\mathbb{H}(N)$ under the maps $\pi_N^{N+1}:\mathbb{H}(N+1)\to \mathbb{H}(N)$ given by:
\begin{equation*}
    \pi_N^{N+1}\left[\left(\mathbf{H}_{ij}\right)_{i,j=1}^{N+1}\right]=\left(\mathbf{H}_{ij}\right)_{i,j=1}^N.
\end{equation*}
Moreover, define the maps $\pi_N^\infty:\mathbb{H}\to \mathbb{H}(N)$ by $\pi_N^{\infty}\left[\left(\mathbf{H}_{ij}\right)_{i,j=1}^{\infty}\right]=\left(\mathbf{H}_{ij}\right)_{i,j=1}^N$.

We have the natural action $\mathbb{U}(\infty)\curvearrowright\mathbb{H}$ by conjugation. Denote by $\mathcal{M}_p^{\textnormal{inv}}\left(\mathbb{H}\right)$ the space of invariant (under this action) probability measures on $\mathbb{H}$ and by $\mathcal{M}_p^{\textnormal{erg}}\left(\mathbb{H}\right)$ the space of ergodic measures. It is a classical result of Pickrell \cite{Pickrell} and Olshanski and Vershik \cite{OlshanskiVershik} that $\mathcal{M}_p^{\textnormal{erg}}\left(\mathbb{H}\right)$ is in bijection with the space $\Omega$, see \cite{Pickrell,OlshanskiVershik} for a precise statement. We write $\mathfrak{N}_\omega$ for the ergodic measure corresponding to $\omega \in \Omega$. In fact, $\mathfrak{N}_\omega$ has a nice explicit construction \cite{OlshanskiVershik}. Consider, for $\omega \in \Omega$, the random matrix $\mathbf{H}^{(\omega)}\in \mathbb{H}$: 
\begin{equation}\label{ExtremalHermitian}
\mathbf{H}^{(\omega)}_{ij}=\gamma_1 \mathbf{1}_{i=j}+\sqrt{\gamma_2}\mathbf{G}_{ij}+\sum_{k=1}^\infty\alpha_k^+\left(\xi_i^{+,(k)}\overline{\xi_j^{+,(k)}}-\mathbf{1}_{i=j}\right)+\sum_{k=1}^\infty\left(-\alpha_k^-\right)\left(\xi_i^{-,(k)}\overline{\xi_j^{-,(k)}}-\mathbf{1}_{i=j}\right),
\end{equation}
where $\mathbf{G}$ is an infinite Gaussian Unitary Ensemble matrix (with normalisation $\mathbb{E}\left[\mathbf{G}^2_{ii}\right]=1$) and $\{\xi_{i}^{+,(j)}\}_{i,j=1}^\infty, \{\xi_{i}^{-,(j)}\}_{i,j=1}^\infty $ are two independent families of independent standard complex Gaussian random variables. Then, $\mathfrak{N}_\omega$ is simply the law of $\mathbf{H}^{(\omega)}$. Moreover, it was later shown by Borodin and Olshanski that any $\mathsf{M}\in \mathcal{M}_p^{\textnormal{inv}}\left(\mathbb{H}\right)$ can be decomposed into ergodic measures as follows:
\begin{equation}\label{DecompositionMatrices}
 \mathsf{M}=\int_{\Omega} \mathfrak{N}_\omega\nu_{\mathsf{M}}(d\omega),
 \end{equation}
for a unique $\nu_{\mathsf{M}}\in \mathcal{M}_p(\Omega)$, see \cite{BorodinOlshanskiErgodic} for the precise statement\footnote{Observe that, in principle this gives a concrete matrix model for any $\mathsf{M}\in \mathcal{M}_p^{\textnormal{inv}}\left(\mathbb{H}\right)$. Simply take $\omega$ in (\ref{ExtremalHermitian}) random, distributed according to $\nu_{\mathsf{M}}$, and independent of $\mathsf{G},\{\xi_{i}^{+,(j)}\}_{i,j=1}^\infty, \{\xi_{i}^{-,(j)}\}_{i,j=1}^\infty $.}. Finally, the map (\ref{DecompositionMatrices}) is in fact a bijection between $\mathcal{M}_p^{\textnormal{inv}}\left(\mathbb{H}\right)$ and $\mathcal{M}_p\left(\Omega\right)$, see \cite{BorodinOlshanskiErgodic}.

We need a final piece of notation. For any $N\ge 1$, define the function $\mathsf{eval}: \mathbb{H}(N)\to \mathbb{W}_N$, where $\mathbb{W}_N$ is the Weyl chamber (\ref{WeylChamber}), which takes a matrix $\mathbf{H}$ to its ordered eigenvalues in a non-decreasing fashion counted with multiplicity. We have now arrived at the main result of this section.

\begin{thm}\label{HermitianTheorem}
Let $\mathsf{M}\in \mathcal{M}_p^{\textnormal{inv}}(\mathbb{H})$ be the law of $\mathbf{H}\in \mathbb{H}$. Then, $\mathsf{M}$-a.s. we have uniformly on compact sets in $\mathbb{C}$:
\begin{equation}
  \Psi_N\left(\frac{z}{N}\right)=\prod_{i=1}^N\left(1-z\frac{\mathsf{eval}\left(\pi_N^\infty\left(\mathbf{H}\right)\right)}{N}\right)=\det\left(\mathbf{I}-\frac{z}{N}\pi_N^\infty\left(\mathbf{H}\right)\right)  \overset{N\to\infty}{\longrightarrow} \mathsf{E}_\omega(z) , \textnormal{ where } \mathsf{Law}(\omega)=\nu_{\mathsf{M}}.
\end{equation}
Moreover, any random entire function in $\mathcal{LP}$ can be obtained in this way.
\end{thm}

\begin{proof}
From the results of Section 5 of \cite{BorodinOlshanskiErgodic} (building on the results of \cite{OlshanskiVershik}) we have that $\left\{\frac{\mathsf{eval}\left(\pi_N^\infty\left(\mathbf{H}\right)\right)}{N}\right\}_{N=1}^\infty$ satisfies the (O-V) conditions $\mathsf{M}$-a.s. and moreover if $\omega$ denotes the random limit point then $\mathsf{Law}(\omega)=\nu_{\mathsf{M}}$. Thus we apply Proposition \ref{ConvergenceProp} to get the first statement. For the second statement suppose we are given a random $\mathsf{E}_\omega \in \mathcal{LP}$ with $\mathsf{Law}(\omega)=\nu\in \mathcal{M}_p\left(\Omega\right)$ arbitrary. We then consider $\mathsf{M}\overset{\textnormal{def}}{=}\int_\Omega \mathfrak{N}_\omega \nu(d\omega)\in \mathcal{M}_p^{\textnormal{inv}}(\mathbb{H})$ and the conclusion follows from the first statement just proven.
\end{proof}

\begin{rmk}
As noted earlier, we use in the proof above that a random entire function in $\mathcal{LP}$ is equivalent to a probability measure on $\Omega$, which itself is proved in Proposition \ref{MeasureProp}.
\end{rmk}

Since $\nu_{\mathfrak{N}_\omega}$ is simply the delta measure at $\omega$ we immediately get the following corollary:

\begin{cor}
Let $\mathsf{E}_\omega$ be an arbitrary deterministic function in $\mathcal{LP}$. Then, $\mathfrak{N}_\omega$-a.s. we have uniformly on compact sets in $\mathbb{C}$:
\begin{equation*}
  \det\left(\mathbf{I}-\frac{z}{N}\pi_N^\infty\left(\mathbf{H}^{(\omega)}\right)\right)   \overset{N\to\infty}{\longrightarrow} \mathsf{E}_\omega(z).
\end{equation*}
\end{cor}

\subsection{The stochastic Bessel and Hua-Pickrell functions}\label{DistinguishedCasesBeta=2}

We consider two distinguished cases of random entire functions. For any $N\ge 1$ and $s \in \mathbb{C}$ with $\Re(s)>-\frac{1}{2}$ we define the probability measure $\mathfrak{M}_N^{\textnormal{HP},s}$ on $\mathbb{H}(N)$:
\begin{equation*}
\mathfrak{M}_N^{\textnormal{HP},s}(d\mathbf{H})\propto  \det\left(\mathbf{I}+i\mathbf{H}\right)^{-s-N} \det\left(\mathbf{I}-i\mathbf{H}\right)^{-\bar{s}-N} d\mathbf{H},
\end{equation*}
where $d\mathbf{H}$ denotes Lebesgue measure on $\mathbb{H}(N)$ and $\propto$ denotes proportionality. These measures are called the Hua-Pickrell or Cauchy measures  \cite{BorodinOlshanskiErgodic,Neretin,ForresterBook}. The implicit normalisation constants for all probability measures we consider in this paper are explicit and can be found for example in \cite{ForresterBook}. Moreover, for any $N\ge 1$ and for $\eta>-1$ we define the probability measure $\mathfrak{M}_N^{\textnormal{L},\eta}$ on $\mathbb{H}_+(N)\subset \mathbb{H}(N)$, the space of $N\times N$ non-negative definite Hermitian matrices:
\begin{equation*}
\mathfrak{M}_N^{\textnormal{L},\eta}(d\mathbf{H}) \propto \det\left(\mathbf{H}\right)^{\eta} \exp \left(-\textnormal{Tr}\mathbf{H}\right)\mathbf{1}_{\mathbf{H}\in \mathbb{H}_+(N)} d\mathbf{H}.
\end{equation*}
These are called the Laguerre or Wishart measures \cite{ForresterBook}. Under the transformation $\mathbf{H}\mapsto 2 \mathbf{H}^{-1}$ we obtain the so-called inverse Laguerre measures:
\begin{equation*}
\mathfrak{M}_N^{\textnormal{IL},\eta}(d\mathbf{H}) \propto \det\left(\mathbf{H}\right)^{-\eta-2N} \exp \left(-2\textnormal{Tr}\mathbf{H}^{-1}\right)\mathbf{1}_{\mathbf{H}\in \mathbb{H}_+(N)} d\mathbf{H}.
\end{equation*}
Observe that all these measures are unitarily invariant. Moreover, they have the remarkable property that they are consistent, see \cite{BorodinOlshanskiErgodic,Neretin,NeretinRayleigh,ErgodicDecompositionInverseWishart} for different methods of proof:
\begin{equation*}
  \left(\pi_N^{N+1}\right)_*\mathfrak{M}_{N+1}^{\textnormal{HP},s} =\mathfrak{M}_N^{\textnormal{HP},s}  , \ \  \left(\pi_N^{N+1}\right)_*\mathfrak{M}_{N+1}^{\textnormal{IL},\eta}=\mathfrak{M}_N^{\textnormal{IL},\eta},\ \ \forall N\ge 1.
\end{equation*}
Hence, from Kolmogorov's theorem, see for example Chapter 6 in \cite{Kallenberg}, we obtain unique $\mathfrak{M}^{\textnormal{HP},s}, \mathfrak{M}^{\textnormal{IL},\eta}\in \mathcal{M}_p^{\textnormal{inv}}\left(\mathbb{H}\right)$ having the correct projections on the $N\times N$ principal submatrices:
\begin{equation*}
\left(\pi_N^{\infty}\right)_*\mathfrak{M}^{\textnormal{HP},s} =\mathfrak{M}_N^{\textnormal{HP},s}  , \ \  \left(\pi_N^{\infty}\right)_*\mathfrak{M}^{\textnormal{IL},\eta}=\mathfrak{M}_N^{\textnormal{IL},\eta},\ \ \forall N\ge 1.\end{equation*}

Write $\mathsf{SHP}_s(z)$ and $\mathsf{SB}_\eta(z)$ for (a realisation) of the random entire functions $\mathsf{E}_\omega$ that we obtain from Theorem \ref{HermitianTheorem} for $\mathfrak{M}^{\textnormal{HP},s}$ and $\mathfrak{M}^{\textnormal{IL},\eta}$ respectively. Our goal is to understand their law, namely $\mathsf{Law}\left(\omega\right)$. Write\footnote{In the notation of the previous subsection these are given by $\nu_{\mathfrak{M}^{\textnormal{HP},s}}$ and $\nu_{\mathfrak{M}^{\textnormal{IL},\eta}}$.} $\nu_{\textnormal{HP}}^{s}$ and $\nu_{\textnormal{IL}}^{\eta}$ for these respectively. 

We first consider $\mathsf{SB}_\eta(z)$. By non-negativity of eigenvalues we obtain, see Remark \ref{RemarkPositive}, $\alpha_i^-\equiv 0$ for all $i\ge 1$ and $\gamma_2\equiv 0$, $\nu_{\textnormal{IL}}^{\eta}$-a.s. Moreover, it was shown in \cite{ErgodicDecompositionInverseWishart} that $\gamma_1=\sum_{i=1}^\infty \alpha_i^+$,  $\nu_{\textnormal{IL}}^{\eta}$-a.s. Finally, under $\nu_{\textnormal{IL}}^{\eta}$ the law of the $\alpha^+$ parameters coincides with the law of the reciprocals of points $\mathfrak{b}_\eta(1)<\mathfrak{b}_\eta(2)<\mathfrak{b}_\eta(3)<\cdots$ of the determinantal point process with the Bessel kernel $\mathsf{K}_{\textnormal{Bes}}^\eta$, see \cite{ErgodicDecompositionInverseWishart} for the explicit expression. In particular, the law of  $\mathsf{SB}_{\eta}(z)$ takes the especially nice form:
\begin{equation*}
 \mathsf{SB}_{\eta}(z) \overset{\textnormal{d}}{=} \prod_{i=1}^{\infty}\left(1-\frac{z}{\mathfrak{b}_\eta(i)}\right).
\end{equation*}
This is the function (up to matching of the parameters) obtained in \cite{ValkoLin} as the scaling limit of the characteristic polynomial of a different ensemble, the real orthogonal $\beta$-ensemble, for $\beta=2$. This function has connection to integrable systems, in particular the characteristic function of $\gamma_1$ was studied in \cite{DistinguishedFamily} and (a transformation of it) was shown to solve a special case of  the $\sigma$-Painlev\'{e} III' equation. Finally, a combinatorial expression for the (finite) even moments of $\gamma_1$ was obtained in \cite{CircularJacobiJoint}.

We turn our attention to $\mathsf{SHP}_s(z)$. It was shown in \cite{Qiu} that $\gamma_2=0$, $\nu_{\textnormal{HP}}^{s}$-a.s. for any $s$. Moreover, it was proven in the same paper that $\gamma_1$ is given as a principal value sum:
\begin{equation}\label{PrincipalValue}
    \gamma_1=\lim_{R\to \infty}\left(\sum_{i=1}^\infty \alpha_i^+\mathbf{1}_{\alpha_i^+>R^{-2}}-\sum_{i=1}^\infty\alpha_i^-\mathbf{1}_{\alpha_i^->R^{-2}}\right),
\end{equation} $\nu_{\textnormal{HP}}^{s}$-a.s. for real $s$ (but the same result is expected to hold for complex $s$). Finally, under $\nu_{\textnormal{HP}}^{s}$ the law of the $\alpha^+$ and $-\alpha^-$ parameters (viewed as a point process) coincides with the law of the points of the determinantal point process with a kernel $\mathsf{K}_{\textnormal{HP}}^s$  given in terms of hypergeometric functions, see \cite{BorodinOlshanskiErgodic}. Moreover, in the special case $s=0$ we have that $\{-(\pi\alpha_i^+)^{-1}\}\sqcup\{(\pi\alpha_i^-)^{-1}\}$ is distributed as the determinantal point process with the sine kernel, see again \cite{BorodinOlshanskiErgodic}. Then, the function $\mathsf{SHP}_0\left(-\frac{z}{\pi}\right)$ matches (up to a $e^{i \pi z}$ factor) in distribution the function constructed by Chhaibi-Najnudel-Nikeghbali in \cite{ChhaibiNajnudelNikeghbali}. This is presented in \cite{ChhaibiNajnudelNikeghbali} as a principal value product but its equivalence to an $\mathsf{E}_\omega$-type formula follows from Proposition \ref{PrincipalValProp} in the sequel. Distributional properties of this function were studied in \cite{ChhaibiNajnudelNikeghbali,ValkoViragOperatorLimit}. Analogous comparisons can be made for $s\neq 0$ and $\beta=2$ with the functions constructed in \cite{ValkoLin}. Finally, again these functions have connections to integrable systems: the characteristic function (more precisely a transformation thereof) of $\gamma_1$ for $s\in \mathbb{R}$ was shown in \cite{DistinguishedFamily} to solve a special case of  $\sigma$-Painlev\'{e} III' equation. Also in the same paper explicit expressions for this distribution were obtained for $s\in \mathbb{N}$.

\section{Applications to consistent $\beta$-ensembles}\label{SectionBeta}

\subsection{A general convergence theorem for consistent $\beta$-ensembles}

The results of the previous section have a natural extension to $\beta$-ensembles. The correct general $\beta$ analogue of random unitarily invariant infinite Hermitian matrices is that of consistent random infinite interlacing arrays that we define next. For the special values $\beta=1,2,4$ this model is equivalent (by looking at the eigenvalues of consecutive principal submatrices) to infinite self-adjoint random matrices with real entries for $\beta=1$, or complex entries for $\beta=2$ or quaternion entries for $\beta=4$ and whose law is invariant under orthogonal for $\beta=1$, or unitary for $\beta=2$ or symplectic for $\beta=4$ conjugation respectively, see Proposition 1.7 in \cite{OrbitalBeta} for more details.

We begin with some definitions. We say that $\mathbf{x}\in \mathbb{W}_N$ and $\mathbf{y}\in \mathbb{W}_{N+1}$ interlace and write $\mathbf{x}\prec \mathbf{y}$ if the following inequalities hold:
\begin{equation*}
y_1\ge x_1 \ge y_2 \ge x_2 \ge \cdots \ge y_N\ge x_N \ge y_{N+1}.
\end{equation*}
We call a sequence $\left\{\mathbf{x}^{(N)} \right\}_{N=1}^\infty$in $\left\{\mathbb{W}_N\right\}_{N=1}^\infty$ so that $\mathbf{x}^{(1)}\prec \mathbf{x}^{(2)}\prec \mathbf{x}^{(3)} \prec \cdots$ an infinite interlacing array and write $\mathfrak{IA}$ for the space of all of such arrays. For $N\ge 1$, we consider the Markov kernel $\mathsf{\Lambda}_{N+1,N}^{(\beta)}$ from $\mathbb{W}_{N+1}$ to $\mathbb{W}_{N}$ defined as follows:

\begin{defn} Let $\beta>0$. For $\mathbf{y}\in \mathbb{W}_{N+1}$, 
$\mathsf{\Lambda}_{N+1,N}^{(\beta)}\left(\mathbf{y},\cdot\right)$ is the distribution of the non-increasing roots counted with multiplicity of the random polynomial:
\begin{equation*}
 z\mapsto  \sum_{j=1}^{N+1} \mathfrak{d}_j \prod_{1\le k \le N, k \neq j} (z-y_k),
\end{equation*}
where the vector $\left(\mathfrak{d}_1,\dots,\mathfrak{d}_{N+1} \right)$ is Dirichlet distributed with all parameters equal to $\beta/2$. We note that $\mathsf{\Lambda}_{N+1,N}^{(\beta)}\left(\mathbf{y},\cdot\right)$ is supported on $\mathbf{x}\in \mathbb{W}_N$ such that $\mathbf{x}\prec \mathbf{y}$, see \cite{ForresterBook,OrbitalBeta}.
\end{defn}
 $\mathsf{\Lambda}_{N+1,N}^{(\beta)}$ also has an equivalent explicit expression in terms of the Dixon-Anderson conditional probability distribution, see \cite{ForresterBook,OrbitalBeta}. We need the following definition:
 
 \begin{defn}
 Let $\beta>0$. We say that a random infinite interlacing array $\left\{\mathbf{x}^{(N)} \right\}_{N=1}^\infty$ is consistent if for all $N\ge 1$, the distribution of the first $N$ rows $\left(\mathbf{x}^{(1)},\dots,\mathbf{x}^{(N)}\right)$ is given by:
 \begin{equation*}
m_N\left(d\mathbf{x}^{(N)}\right)\mathsf{\Lambda}_{N,N-1}^{(\beta)}\left(\mathbf{x}^{(N)},d\mathbf{x}^{(N-1)}\right)\mathsf{\Lambda}_{N-1,N-2}^{(\beta)}\left(\mathbf{x}^{(N-1)},d\mathbf{x}^{(N-2)}\right)\cdots \mathsf{\Lambda}_{2,1}^{(\beta)}\left(\mathbf{x}^{(2)},d\mathbf{x}^{(1)}\right),
 \end{equation*}
 where $m_N=\mathsf{Law}\left(\mathbf{x}^{(N)}\right)$. We denote the set of all consistent distributions on $\mathfrak{IA}$ by $\mathcal{M}_p^{\textnormal{c},(\beta)}\left(\mathfrak{IA}\right)$.
 \end{defn}

Recall that $\mathsf{M}^{(\beta)}\in \mathcal{M}_p^{\textnormal{c},(\beta)}\left(\mathfrak{IA}\right)$ is extremal if $\mathsf{M}^{(\beta)}=t\mathsf{M}^{(\beta)}_1+(1-t)\mathsf{M}^{(\beta)}_2$, with $t\in (0,1)$ and $\mathsf{M}^{(\beta)}_1,\mathsf{M}^{(\beta)}_2\in \mathcal{M}_p^{\textnormal{c},(\beta)}\left(\mathfrak{IA}\right)$, implies that $\mathsf{M}^{(\beta)}_1=\mathsf{M}^{(\beta)}_2=\mathsf{M}^{(\beta)}$. We write $\textnormal{Ex}(\mathcal{M}_p^{\textnormal{c},(\beta)}\left(\mathfrak{IA}\right))$  for these extreme points of $\mathcal{M}_p^{\textnormal{c},(\beta)}\left(\mathfrak{IA}\right)$.

Finally, for $\beta>0$, we say that an infinite sequence of probability measures $\left\{\mu_N^{(\beta)}\right\}_{N=1}^\infty$ with $\mu_N^{(\beta)}\in \mathcal{M}_p\left(\mathbb{W}_N\right)$ is consistent (with parameter $\beta$) if:
 \begin{equation}\label{ConsistentMeasures}
   \mu_{N+1}^{(\beta)}\mathsf{\Lambda}_{N+1,N}^{(\beta)}=\mu_N^{(\beta)} , \ \ \forall N\ge 1.
 \end{equation}
By Kolmogorov's theorem consistent sequences of probability measures $\left\{\mu_N^{(\beta)}\right\}_{N=1}^\infty$ are in bijection with $\mathcal{M}_p^{\textnormal{c},(\beta)}\left(\mathfrak{IA}\right)$. In particular, the corresponding consistent distribution $\mathsf{M}^{(\beta)}$ on $\mathfrak{IA}$ provides a natural coupling.

The extremal consistent distributions on $\mathfrak{IA}$ were classified in \cite{OrbitalBeta}, making use of some earlier results from \cite{CuencaOrbital}. It was shown that $\textnormal{Ex}(\mathcal{M}_p^{\textnormal{c},(\beta)}\left(\mathfrak{IA}\right))$ is in correspondence with $\Omega$, see \cite{OrbitalBeta} for the precise statement. We denote an extremal measure parametrised by $\omega \in \Omega$ by $\mathfrak{N}_\omega^{(\beta)}$. These are no longer explicit as in (\ref{ExtremalHermitian}) for $\beta=2$, but they are characterised by their explicit Dunkl transform, see \cite{OrbitalBeta}. Moreover, it was shown that any $\mathsf{M}^{(\beta)}\in \mathcal{M}_p^{\textnormal{c},(\beta)}\left(\mathfrak{IA}\right)$ can be decomposed into extremal measures as follows:
\begin{equation}\label{DecompositionInter}
 \mathsf{M}^{(\beta)}=\int_{\Omega} \mathfrak{N}_\omega^{(\beta)}\nu_{\mathsf{M}^{(\beta)}}(d\omega),
 \end{equation}
for a unique $\nu_{\mathsf{M}^{(\beta)}}\in \mathcal{M}_p(\Omega)$, see \cite{OrbitalBeta}. The map (\ref{DecompositionInter}) is actually a bijection between $\mathcal{M}_p^{\textnormal{c},(\beta)}\left(\mathfrak{IA}\right)$ and $\mathcal{M}_p(\Omega)$, see \cite{OrbitalBeta}. We have the following generalisation of Theorem \ref{HermitianTheorem}.

\begin{thm}\label{InterlacingTheorem}
Let $\beta>0$. Let $\mathsf{M}^{(\beta)}\in \mathcal{M}_p^{\textnormal{c},(\beta)}\left(\mathfrak{IA}\right)$ and denote by $\left\{\mathsf{x}^{(N)}\right\}_{N=1}^\infty$ a random interlacing array with $\mathsf{Law}\left(\left\{\mathsf{x}^{(N)}\right\}_{N=1}^\infty\right)=\mathsf{M}^{(\beta)}$. Then, $\mathsf{M}^{(\beta)}$-a.s. we have uniformly on compact sets in $\mathbb{C}$:
\begin{equation}
  \Psi_{N}\left(\frac{z}{N}\right)=\prod_{i=1}^N\left(1-z\frac{x_i^{(N)}}{N}\right) \overset{N\to\infty}{\longrightarrow} \mathsf{E}_\omega(z) , \textnormal{ where } \mathsf{Law}(\omega)=\nu_{\mathsf{M}^{(\beta)}}.
\end{equation}
Moreover, any random entire function in $\mathcal{LP}$ can be realised in this way.
\end{thm}

\begin{proof}
From Theorem 3.6 in \cite{OrbitalBeta} we have that the sequence $\left\{\mathsf{x}^{(N)}/N\right\}_{N=1}^\infty$ satisfies the (O-V) conditions $\mathsf{M}^{(\beta)}$-a.s. and moreover if $\omega$ denotes the random limit point, $\mathsf{Law}(\omega)=\nu_{\mathsf{M}^{(\beta)}}$. We then apply Proposition \ref{ConvergenceProp} to get the first statement. For the second statement suppose we are given a random $\mathsf{E}_\omega \in \mathcal{LP}$ with $\mathsf{Law}(\omega)=\nu\in \mathcal{M}_p\left(\Omega\right)$ arbitrary. We then consider $\mathsf{M}\overset{\textnormal{def}}{=}\int_\Omega \mathfrak{N}^{(\beta)}_\omega \nu(d\omega)\in  \mathcal{M}_p^{\textnormal{c},(\beta)}\left(\mathfrak{IA}\right)$ and the conclusion follows from the first statement just proven.
\end{proof}

\subsection{The stochastic Bessel and Hua-Pickrell functions for general $\beta$}

We briefly consider the general $\beta$-ensemble versions of the two random entire functions from Section \ref{DistinguishedCasesBeta=2}. For any $N\ge 1$ and $s \in \mathbb{C}$ with $\Re(s)>-\frac{1}{2}$ we define the probability measure $\mathfrak{M}_N^{\textnormal{HP},s,(\beta)}$ on $\mathbb{W}_N$:
\begin{equation*}
\mathfrak{M}_N^{\textnormal{HP},s,(\beta)}(d\mathbf{x})\propto \prod_{j=1}^N\left(1+ix_j\right)^{-s-\beta(N-1)/2-1}\left(1-ix_j\right)^{-\bar{s}-\beta(N-1)/2-1}\prod_{1\le i<j \le N}\left|x_i-x_j\right|^\beta d\mathbf{x}.
\end{equation*}
Moreover, for any $N\ge 1$ and $\eta>-1$ we define the probability measure $\mathfrak{M}_N^{\textnormal{IL}}$ on $\mathbb{W}_N^+=\mathbb{W}_N\cap \mathbb{R}_+^N$, the non-negative Weyl chamber:
\begin{equation*}
\mathfrak{M}_N^{\textnormal{IL},\eta,(\beta)}(d\mathbf{x})\propto \prod_{j=1}^N x_j^{-\eta-(N-1)\beta-2}e^{-\frac{2}{x_j}}\prod_{1\le i<j \le N}\left|x_i-x_j\right|^\beta \mathbf{1}_{\mathbf{x}\in \mathbb{W}_N^+}d\mathbf{x}.
\end{equation*}
The connection with the Hermitian matrix measures $\mathfrak{M}_N^{\textnormal{HP},s}, \mathfrak{M}_N^{\textnormal{IL},\eta}$ from Section \ref{SectionHermitian} is standard using the Weyl integration formula \cite{ForresterBook}:
\begin{equation*}
  \left(\mathsf{eval}\right)_*\mathfrak{M}_N^{\textnormal{HP},s}=\mathfrak{M}_N^{\textnormal{HP},s,(2)}  , \ \ \left(\mathsf{eval}\right)_*\mathfrak{M}_N^{\textnormal{IL},\eta}=\mathfrak{M}_N^{\textnormal{IL},\eta,(2)}, \ \ \forall N\ge 1.
\end{equation*}
Moreover, it follows from Lemma 2.2 in \cite{NeretinRayleigh} that these measures are consistent:
\begin{equation*}
\mathfrak{M}_{N+1}^{\textnormal{HP},\eta,(\beta)}\mathsf{\Lambda}_{N+1,N}^{(\beta)}=\mathfrak{M}_N^{\textnormal{HP},\eta,(\beta)}, \ \ \mathfrak{M}_{N+1}^{\textnormal{IL},\eta,(\beta)}\mathsf{\Lambda}_{N+1,N}^{(\beta)}=\mathfrak{M}_N^{\textnormal{IL},\eta,(\beta)}, \ \ \forall N\ge 1, \ \forall \beta>0.
\end{equation*}
We write $\mathfrak{M}^{\textnormal{HP},s,(\beta)}, \mathfrak{M}^{\textnormal{IL},\eta,(\beta)}\in \mathcal{M}_p^{\textnormal{c},(\beta)}\left(\mathfrak{IA}\right)$ for the corresponding couplings.

Write $\mathsf{SHP}^{(\beta)}_s(z)$ and $\mathsf{SB}^{(\beta)}_\eta(z)$ for (a realisation) of the random entire functions $\mathsf{E}_\omega$ that we obtain from Theorem \ref{InterlacingTheorem} for $\mathfrak{M}^{\textnormal{HP},s,(\beta)}$ and $\mathfrak{M}^{\textnormal{IL},\eta,(\beta)}$ respectively. Clearly, these specialise, for $\beta=2$, to the functions from Section \ref{DistinguishedCasesBeta=2}.  Our goal is to understand their law, namely $\mathsf{Law}\left(\omega\right)$. As before, we write $\nu_{\textnormal{HP}}^{s,(\beta)}$ and $\nu_{\textnormal{IL}}^{\eta,(\beta)}$ for these laws respectively.

We first consider $\mathsf{SB}^{(\beta)}_\eta(z)$. By non-negativity of eigenvalues we obtain from Remark \ref{RemarkPositive} that $\alpha_i^-\equiv 0$ for all $i\ge 1$ and $\gamma_2\equiv 0$, $\nu_{\textnormal{IL}}^{\eta,(\beta)}$-a.s. Moreover, it  follows\footnote{It was shown, in an equivalent form, in \cite{RiderRamirez} that there exists a coupling $\tilde{\mathfrak{M}}^{\textnormal{IL},\eta,(\beta)}$ of the $\mathfrak{M}_N^{\textnormal{IL},\eta,(\beta)}$'s and a subsequence $\{N_k\}_{k=1}^\infty$ such that $\tilde{\mathfrak{M}}^{\textnormal{IL},\eta,(\beta)}$-a.s.:
\begin{equation*}
 \left(\alpha^{+,(N_k)}\left(\mathbf{x}^{(N_k)}\right),\gamma_1^{(N_k)}\left(\mathbf{x}^{(N_k)}\right)\right) \overset{k\to \infty}{\longrightarrow} \left(\alpha^+, \sum_{i=1}^\infty \alpha_i^+\right).
\end{equation*}
We also know from Theorem 3.6 in \cite{OrbitalBeta} that the whole sequence converges $\mathfrak{M}^{\textnormal{IL},\eta,(\beta)}$-a.s. which defines $(\alpha^+,\gamma_1)$. This identifies the joint distribution under $\nu_{\textnormal{IL}}^{\eta,(\beta)}$ of $\left(\alpha^+,\gamma_1\right)$ as that of $\left(\alpha^+,\sum_{i=1}^\infty \alpha_i^+\right)$ and the conclusion follows. } from the results of \cite{RiderRamirez} that $\gamma_1=\sum_{i=1}^\infty \alpha_i^+$,  $\nu_{\textnormal{IL}}^{\eta,(\beta)}$-a.s. Finally, under $\nu_{\textnormal{IL}}^{\eta,(\beta)}$ the law of the $\alpha^+$ parameters is given by the law of reciprocals of random eigenvalues $\mathfrak{b}^{(\beta)}_\eta(1)<\mathfrak{b}^{(\beta)}_\eta(2)<\mathfrak{b}^{(\beta)}_\eta(3)<\cdots$ of a stochastic operator $\mathfrak{B}_\eta^{(\beta)}$ (which has trace class inverse), see \cite{RiderRamirez} for the precise description. Hence, as before the law of  $\mathsf{SB}^{(\beta)}_{\eta}(z)$ takes the particularly nice form:
\begin{equation*}
 \mathsf{SB}^{(\beta)}_{\eta}(z) \overset{\textnormal{d}}{=} \prod_{i=1}^{\infty}\left(1-\frac{z}{\mathfrak{b}^{(\beta)}_\eta(i)}\right).
\end{equation*}
This is the function (up to matching of the parameters) obtained in \cite{ValkoLin} as the limit of the characteristic polynomial of the real orthogonal $\beta$-ensemble. Finally, in terms of explicit formulae a combinatorial expression for the even moments of $\gamma_1$ was obtained in \cite{CircularJacobiJoint}.

We then turn to $\mathsf{SHP}_s^{(\beta)}(z)$. First, it follows from the results of \cite{ValkoVirag,ValkoViragOperatorLimit,ValkoLin} that the (joint) distribution of the $\alpha^+$ and $-\alpha^-$ parameters, under $\nu_{\textnormal{HP}}^{s,(\beta)}$, is given by the law of the random eigenvalues of a stochastic operator, see \cite{ValkoVirag,ValkoViragOperatorLimit,ValkoLin} precise statements. This is the general $\beta>0$ extension of the determinantal point process with kernel $\mathsf{K}^s_{\textnormal{HP}}$. The distinguished case $s=0$, under the map (up to multiplicative constant) $x\mapsto x^{-1}$ gives the $\beta$-Sine point process, see \cite{ValkoVirag,ValkoViragOperatorLimit}.  Regarding the parameter $\gamma_2$, it should follow\footnote{See in particular Corollary 15 in \cite{ValkoLin}. We note that, the authors in \cite{ValkoLin} work with the equivalent model, called the circular Jacobi $\beta$-ensemble, obtained under the Cayley transform.  We do not discuss the details here. } from the results of \cite{ValkoViragZeta,ValkoLin} that $\gamma_2\equiv 0$, $\nu_{\textnormal{HP}}^{s,(\beta)}$-a.s. Moreover, $\gamma_1$ should be given by the principal value sum formula (\ref{PrincipalValue}), $\nu_{\textnormal{HP}}^{s,(\beta)}$-a.s. In the special case $s=0$ this should be a consequence\footnote{Again, the authors in \cite{ValkoViragZeta} consider the equivalent model of the circular $\beta$-ensemble obtained under an application of the Cayley transform. We do not discuss the details here. } of the results of \cite{ValkoViragZeta}. For $s\neq 0$ it appears that this does not follow from known results. However, it might be possible to extend the proof of \cite{ValkoViragZeta} using the results of \cite{ValkoLin}; this would also give a principal value product formula for $\mathsf{SHP}_s^{(\beta)}(z)$ from Proposition \ref{PrincipalValProp}, see also Remark 24 in \cite{ValkoLin}. We finally note that the distribution of $\gamma_1$ for $s\in \mathbb{N}$ is explicit, see \cite{ForresterJoint}, and explicit formulae for its even moments for any $s$ can be found in \cite{ForresterJoint,CircularJacobiJoint}.

\begin{rmk}\label{Universality}
For any $N\ge 1$, consider the following probability measures on $\mathbb{W}_N^+$:
\begin{equation*}
 \mathfrak{m}^{\mathsf{V},\eta,(\beta)}_N(d\mathbf{x}) \propto \prod_{i=1}^N x_i^\eta e^{-\beta N \mathsf{V}(x_i)} \prod_{1\le i < j\le N}\left|x_i-x_j\right|^\beta \mathbf{1}_{\mathbf{x}\in \mathbb{W}_N^+} d\mathbf{x},
\end{equation*}
where $\mathsf{V}$ is a polynomial so that $x\mapsto \mathsf{V}(x^2)$ is uniformly convex and $\eta,\beta$ satisfy certain restrictions, see \cite{UniversalityStochasticBessel}. These general measures are not (necessarily) consistent, so we do not immediately obtain the (O-V) conditions. However, it is shown in \cite{UniversalityStochasticBessel} that there exists a coupling $\mathfrak{m}^{\mathsf{V},\eta,(\beta)}$ of these measures so that $\mathfrak{m}^{\mathsf{V},\eta,(\beta)}$-a.s. the appropriately rescaled points converge to the random eigenvalues $\mathfrak{b}^{(\beta)}_\eta(1)<\mathfrak{b}^{(\beta)}_\eta(2)<\mathfrak{b}^{(\beta)}_\eta(3)<\cdots$ of the stochastic operator $\mathfrak{B}_\eta^{(\beta)}$. Moreover, under a subsequence the sum of reciprocals of points also converges to the trace of the inverse of  $\mathfrak{B}_\eta^{(\beta)}$. In particular, by virtue of Proposition \ref{ConvergenceProp} and Remark \ref{RemarkPositive}, under this subsequence, we obtain that the characteristic polynomial of the appropriately rescaled points converges to $\mathsf{SB}^{(\beta)}_\eta(z)$. Such a universality result should also hold for the stochastic Airy function of \cite{StochasticAiryFunction}, discussed in the next subsection, if one looks at the renormalised characteristic polynomials of the models considered in \cite{UniversalityStochasticAiry}.
\end{rmk}

\subsection{$\mathcal{LP}$ functions with $\beta$-Airy point process zeros}\label{SectionAiry}

In this subsection we study entire functions in $\mathcal{LP}$ with zeros given by the $\beta$-Airy point process which arises as the universal scaling limit at the soft edge of random matrices. The situation is quite a bit more subtle than before\footnote{As far as we know, there is no canonical consistent family of $\beta$-ensembles, like $\mathfrak{M}_N^{\textnormal{HP},s,(\beta)}$ and $\mathfrak{M}_N^{\textnormal{IL},\eta,(\beta)}$, from which the $\beta$-Airy point process arises under the $N^{-1}$ scaling.}.

For any $N\ge 1$, define the following probability measure on $\mathbb{W}_N$, called the Gaussian $\beta$-ensemble (G$\beta$E):
\begin{equation}
  \mathfrak{M}_{\textnormal{G},N}^{(\beta)}(d\mathbf{x})\propto e^{-\sum_{i=1}^N \beta N x^2_i} \prod_{1\le i < j \le N}|x_i-x_j|^\beta d\mathbf{x}.
\end{equation}
It is known that the points $y_i^{(N)}=2N^{\frac{2}{3}}(x_i^{(N)}-1)$, with $\mathbf{x}^{(N)}$ distributed according to $\mathfrak{M}_{\textnormal{G},N}^{(\beta)}$, converge in distribution to the eigenvalues $\mathfrak{a^{(\beta)}}(1)>\mathfrak{a}^{(\beta)}(2)>\mathfrak{a}^{(\beta)}(3)>\cdots$ of a stochastic operator $\mathsf{A}_{(\beta)}$ formally written as:
\begin{equation*}
   \mathsf{A}_{(\beta)}=\frac{d^2}{dx^2}-x-\frac{2}{\sqrt{\beta}}d\mathsf{B}_x,
\end{equation*}
where $d\mathsf{B}_x$ denotes white noise, see \cite{StochasticAiry} for the precise statement. Moreover, the inverse operator $\mathsf{A}_{(\beta)}^{-1}$ is a.s. well-defined and also Hilbert-Schmidt, see \cite{HardToSoft}. The point process of random eigenvalues of $\mathsf{A}_{(\beta)}$ is called the $\beta$-Airy point process.

One might then hope that the following characteristic polynomial converges:
\begin{equation*}
   \prod_{i=1}^N\left(1-\frac{z}{y_i^{(N)}}\right)\overset{?}{\longrightarrow},
\end{equation*}
but this is not the case! It is easy to see that the $\gamma^{(N)}_1$ parameter, namely the sum of points $(y_i^{(N)})^{-1}$, diverges. Some heuristics\footnote{If we assume that we can use the test function $f(x)=\frac{1}{2}(1-x)^{-1}$ in the standard convergence of the empirical measure of G$\beta$E to the semicircle law on $[-1,1]$ (of course $f$ has a singularity at the endpoint of the support), we then have:
\begin{equation*}
\sum_{i=1}^{N}\frac{1}{y_i^{(N)}}=-N^{\frac{1}{3}}\frac{1}{N}\sum_{i=1}^N\frac{1}{2(1-\lambda_i^{(N)})}\sim -N^{\frac{1}{3}}\int_{-1}^1 \frac{1}{2}(1-x)^{-1}\frac{2}{\pi}\sqrt{1-x^2}dx=-N^{\frac{1}{3}}.
\end{equation*}
Naively, one could try to push this heuristic further to predict Gaussian fluctuations but the resulting variance is infinite. Determining the distribution of $\gamma_1$ appears to be more delicate.} predict that it should grow as $-N^{\frac{1}{3}}$. We might then hope that if we subtract this term it will converge\footnote{This is indeed the case and it is a consequence of the results below. We do not have a direct proof however.}. Finally, the $\delta^{(N)}$ parameter, namely the sum of $(y_i^{(N)})^{-2}$, should converge\footnote{This is again the case and it is a consequence of the results below. We do not have a direct proof however.} and in fact be equal\footnote{This more refined claim does not follow from the results below. If it were proven we would also obtain a positive answer to Question 2 of \cite{StochasticAiryFunction} on whether the stochastic Airy function discussed shortly is a.s. of order $\frac{3}{2}$. This implication can be seen as follows. From Theorem 7 in Chapter 1 of \cite{LevinEntire} we have that the order of the canonical product of a sequence is equal to its convergence exponent, see \cite{LevinEntire} for this terminology. Then, from the almost sure asymptotics for the $\beta$-Airy point process in Theorem 6.1 of \cite{ViragICM} we get that the convergence exponent in this case is equal to $\frac{3}{2}$. In particular, if one shows that the stochastic Airy function has order strictly less than $2$ (so that $\gamma_2\equiv 0$), then it has order $\frac{3}{2}$.} to the square Hilbert-Schmidt norm of  $\mathsf{A}_{(\beta)}^{-1}$, so that $\gamma_2=0$.
 
Armed with these heuristics we are led to consider the following renormalised polynomial:
\begin{equation*}
    \mathfrak{P}_N(z)\overset{\textnormal{def}}{=}e^{-N^{\frac{1}{3}}z}\prod_{i=1}^N\left(1-\frac{z}{y_i^{(N)}}\right)
\end{equation*}
and ask whether it will converge. This natural question was answered in the affirmative by Lambert and Paquette in \cite{StochasticAiryFunction,Lambert2020strong}, who showed convergence as real-analytic functions. 

Before stating their result we need some notation. We denote by $\mathsf{SAi}^{(\beta)}(z)$ the so-called stochastic Airy function constructed in \cite{StochasticAiryFunction} normalised\footnote{In the notation of \cite{StochasticAiryFunction} we have $\mathsf{SAi}^{(\beta)}(z)=\frac{\textnormal{SAi}_z(0)}{\textnormal{SAi}_{0}(0)}$.}  so that $\mathsf{SAi}^{(\beta)}(0)=1$. This arises from the unique solution, subject to some asymptotic condition, of a certain stochastic equation and it is given as a deterministic functional of a single Brownian path $\left(\mathsf{w}_t;t\ge 0\right)$, see the introduction of \cite{StochasticAiryFunction} for the details. Furthermore, it is proven there that $\mathsf{SAi}^{(\beta)}(z)$ is an entire function. Finally, the authors constructed a certain coupling of the measures $\mathfrak{M}_{\textnormal{G},N}^{(\beta)}$ and the Brownian path $\left(\mathsf{w}_t;t\ge 0\right)$, which gives rise to $\mathsf{SAi}^{(\beta)}(z)$, that we denote by $\mathfrak{M}_{\textnormal{G}}^{(\beta)}$. The following is then a special case of the main result of \cite{StochasticAiryFunction} (see the display before Question 3 in \cite{StochasticAiryFunction}):

\begin{thm}\cite{StochasticAiryFunction}\label{StochasticAiryThm} Let $\beta>0$. Then, $\mathfrak{M}^{(\beta)}_{\textnormal{G}}$-a.s. we have convergence of the function and all its derivatives uniformly on compact sets in $\mathbb{R}$:
\begin{equation}\label{AiryConvergence}
  \mathfrak{P}_N(z) \overset{N\to\infty}{\longrightarrow}  \mathsf{SAi}^{(\beta)}(z), \ \ z\in \mathbb{R}.
\end{equation}
\end{thm}

The theorem in \cite{StochasticAiryFunction} is more general, it involves an additional parameter $t\in \mathbb{R}$ and most importantly it is quantitative. Moreover, the authors in \cite{StochasticAiryFunction} go on to mention that they expect that their result (including the quantitative part) extends to $z\in \mathbb{C}$ but for technical reasons they restricted themselves to a statement about real-analytic functions. We now show that if one only cares about the convergence statement (at least in the form above), without any quantitative information, then a short complex analysis argument suffices to do this. We need the following proposition. The proof is given in Section \ref{SectionProofs}.

\begin{prop}\label{PropositionExtension}
Suppose $\mathbf{x}^{(N)}\in \mathbb{W}_N$ and $\mathsf{c}^{(N)}\in \mathbb{R}$ for all $N\ge 1$. Assume the entire functions
\begin{equation*}
 \mathsf{\Phi}_N(z)=e^{\mathsf{c}^{(N)}z}\prod_{i=1}^N\left(1-zx_i^{(N)}\right),
\end{equation*}
along with their first and second derivatives, converge uniformly in a real neighbourhood of the origin to some entire function $\mathsf{E}$, and its first and second derivatives respectively. Then, this convergence of entire functions holds uniformly on compact sets in $\mathbb{C}$ and moreover $\mathsf{E}$ is in $\mathcal{LP}$.
\end{prop}

We apply Proposition \ref{PropositionExtension}, taking as input Theorem \ref{StochasticAiryThm}, to obtain:

\begin{prop}
In the setting of Theorem \ref{StochasticAiryThm}, the convergence in (\ref{AiryConvergence}) extends to convergence of entire functions on compact sets in $\mathbb{C}$. Moreover,  $\mathsf{SAi}^{(\beta)}(z)$ is a.s. in $\mathcal{LP}$.
\end{prop}

We also offer a more abstract result. We write $\mathcal{M}_p^{\mathfrak{Ai},(\beta)}\left(\Omega\right)$ for the subset of probability measures $\nu\in \mathcal{M}_p\left(\Omega\right)$ such that under $\nu$ the $\alpha^+$ and $-\alpha^-$ parameters have the same law as the spectrum of $\mathsf{A}_{(\beta)}^{-1}$. 

In particular, the distribution of the zeros of the corresponding $\mathcal{LP}$ entire function is exactly the $\beta$-Airy point process. Observe that, there is a distinguished $\nu_{\mathsf{SAi}^{(\beta)}}\in \mathcal{M}_p^{\mathfrak{Ai},(\beta)}\left(\Omega\right)$ giving rise to $\mathsf{SAi}^{(\beta)}(z)$ that should be supported on $\Omega_0=\left\{\omega\in \Omega:\gamma_2=0\right\}$. 

We now show that any random $\mathcal{LP}$ function with $\beta$-Airy point process zeros can be obtained in a unique way as a limit of characteristic polynomials of unitarily invariant matrices without the need of renormalisation as for G$\beta$E.

\begin{prop}
Let $\beta>0$ and suppose $\nu\in \mathcal{M}_p^{\mathfrak{Ai},(\beta)}\left(\Omega\right)$. Then, there exists a unique $\mathsf{M}_{(\nu)}\in \mathcal{M}_p^{\textnormal{inv}}\left(\mathbb{H}\right)$ such that if $\mathbf{H}\in \mathbb{H}$ with $\mathsf{Law}\left(\mathbf{H}\right)=\mathsf{M}_{(\nu)}$, we have $\mathsf{M}_{(\nu)}$-a.s. uniformly on compact sets in $\mathbb{C}$:
\begin{equation*}
 \det\left(\mathbf{I}-\frac{z}{N}\pi_N^\infty\left(\mathbf{H}\right)\right)  \overset{N\to\infty}{\longrightarrow} \mathsf{E}_\omega(z) , \textnormal{ where } \mathsf{Law}(\omega)=\nu   .
\end{equation*}
Moreover, for any $\hat{\beta}>0$ there exists a unique $\mathsf{M}^{(\hat{\beta})}_{(\nu)}\in \mathcal{M}_p^{\textnormal{c},(\hat{\beta})}\left(\mathfrak{IA}\right) $ such that if $\left\{\mathbf{x}^{(N)}\right\}_{N=1}^\infty \in \mathfrak{IA}$ with $\mathsf{Law}\left(\left\{\mathbf{x}^{(N)}\right\}_{N=1}^\infty\right)=\mathsf{M}_{(\nu)}^{(\hat{\beta})}$ then $\mathsf{M}_{(\nu)}^{(\hat{\beta})}$-a.s. we have uniformly on compact sets in $\mathbb{C}$:
\begin{equation*}
   \Psi_{N}\left(\frac{z}{N}\right)=\prod_{i=1}^N\left(1-z\frac{x_i^{(N)}}{N}\right) \overset{N\to\infty}{\longrightarrow} \mathsf{E}_\omega(z) , \textnormal{ where } \mathsf{Law}(\omega)=\nu.
\end{equation*}
\end{prop}

\begin{proof}
For the first statement we apply Theorem \ref{HermitianTheorem}. Uniqueness is a consequence of the way $\mathsf{M}_{(\nu)}$ is constructed through the map (\ref{DecompositionMatrices}) which is a bijection. For the second statement we apply Theorem \ref{InterlacingTheorem}. Uniqueness is similarly a consequence of the way $\mathsf{M}^{(\beta)}_{(\nu)}$ is constructed through (\ref{DecompositionInter}) which is also a bijection. 
\end{proof}

\section{Proofs of complex analysis results}\label{SectionProofs}

\begin{proof}[Proof of Proposition \ref{ConvergenceProp}]
We first prove the only if direction. Observe that, from the (O-V) conditions there exists some (finite) constant $K$ such that:
\begin{align*}
\sup_{N\ge 1} \left|\gamma^{(N)}_1\left(\mathbf{x}^{(N)}\right)\right| \le K, \ \ \sup_{N\ge 1}  \delta^{(N)}\left(\mathbf{x}^{(N)}\right)\le K.
\end{align*}
Moreover, using the inequality $\left|(1-z)e^{z}\right|\le e^{4|z|^2}$, valid for any $z\in \mathbb{C}$ we can bound:
\begin{align}\label{bound}
  \left|\Psi_N(z)e^{\gamma^{(N)}_1\left(\mathbf{x}^{(N)}\right)z}\right|&=\left| \prod_{i=1}^N\left(1-zx_i^{(N)}\right)e^{zx_i^{(N)}}\right| \le e^{4 \delta^{(N)}\left(\mathbf{x}^{(N)}\right)|z|^2},\nonumber\\
  \left|\Psi_N(z)\right|&\le e^{\left|\gamma^{(N)}_1\left(\mathbf{x}^{(N)}\right)\right||z|+4\delta^{(N)}\left(\mathbf{x}^{(N)}\right)|z|^2}\le e^{K(|z|+4|z|^2)},\ \ \forall z \in \mathbb{C}.
\end{align}
Hence, we obtain that $\{\Psi_N(z)\}_{N=1}^\infty$ is uniformly bounded on any compact set in $\mathbb{C}$. Then, by Montel's theorem \cite{Ahlfors} the sequence $\{\Psi_N(z)\}_{N=1}^\infty$ is normal: every subsequence has a subsubsequence converging uniformly on compact sets in $\mathbb{C}$ to some entire function $E(z)$ (which a-priori depends on the subsubsequence). If we show that all these possible limit functions coincide and are in fact equal to $\mathsf{E}_\omega(z)$ then we obtain (\ref{ConvergenceDisplay}).

Towards this end let $\{\Psi_{N_k}(z)\}_{k=1}^\infty$ be an arbitrary converging subsubsequence and let $E(z)$ be the corresponding limit entire function. By Hurwitz's theorem \cite{Ahlfors}  $E(z)$ will have only real zeros, if it has any. Furthermore, note that by the identity theorem for analytic functions \cite{Ahlfors}, since $E(0)=1$, the zeros of $E(z)$ are bounded away from $0$ and in general cannot have any accumulation points (since then $E(z)$ would be identically zero). Let us denote the reciprocals of the zeros by $\pm \beta_i^{\pm}$, with $\beta_i^{\pm}\ge 0$, ordered in the following fashion (by convention, if either sequence $\{\beta_i^-\}$, $\{ \beta_i^+\}$ is finite we append to it infinitely many $0$'s): $-\beta_1^-\le -\beta_2^- \le-\beta_3^-\le \cdots \le \beta_3^+ \le \beta_2^+ \le \beta_1^+$. Recall that the $x_i^{(N_k)}$ are the reciprocals of the roots of $\Psi_{N_k}(z)$. Hence, again by Hurwitz's theorem \cite{Ahlfors} (since in a small disk about a root $\frac{1}{\beta_i^+}, -\frac{1}{\beta_i^-}$ of order $r$, $\Psi_{N_k}(z)$ will have exactly $r$ roots for $k$ large enough) we obtain:
\begin{equation*}
  \lim_{k\to \infty} \alpha_i^{\pm,(N_k)}\left(\mathbf{x}^{(N_k)}\right)=  \beta_i^{\pm}, \ \ \forall i \ge 1.
\end{equation*}
In particular, by the (O-V) conditions for the $\alpha^{\pm}$ parameters we get $\beta_i^{\pm}=\alpha_i^{\pm}$ for all $i\ge 1$.

Now, by Hadamard's factorisation theorem \cite{Ahlfors} $E(z)$ (since from (\ref{bound}) it is of order at most $2$ and $E(0)=1$) has a factorisation in the following form:
\begin{align*}
E(z)=e^{\mathsf{c}_1z+\mathsf{c}_2z^2} \prod_{i=1}^\infty e^{z\alpha_i^+}\left(1-z\alpha_i^+\right)\prod_{i=1}^\infty e^{-z \alpha_i^-}\left(1+z\alpha_i^-\right). 
\end{align*}
Then, using the fact that:
\begin{align*}
    \frac{d}{dz}\Psi_{N_k}(z)\bigg|_{z=0}  \overset{k\to\infty}{\longrightarrow}\frac{d}{dz}E(z)\bigg|_{z=0}, \ \ \frac{d^2}{dz^2}\Psi_{N_k}(z)\bigg|_{z=0}  \overset{k\to\infty}{\longrightarrow}\frac{d^2}{dz^2}E(z)\bigg|_{z=0},
\end{align*}
and the (O-V) conditions for the $\gamma_1$ and $\delta$ parameters, we obtain $\mathsf{c}_1=-\gamma_1$ and $\mathsf{c}_2=-\frac{\gamma_2}{2}$. This gives $E(z)=\mathsf{E}_\omega(z)$ and completes the proof of the only if direction.

The if direction is easy. The convergence of the $\alpha_i^{\pm,(N)}\left(\mathbf{x}^{(N)}\right)$ follows by looking at the roots $\frac{1}{\alpha_i^+},-\frac{1}{\alpha_i^-}$ of $\mathsf{E}_\omega$ and arguing using Hurwitz's theorem as before. Moreover, from:
\begin{align*}
  \frac{d}{dz}\Psi_N(z)\bigg|_{z=0}  \overset{N\to\infty}{\longrightarrow}\frac{d}{dz}\mathsf{E}_\omega(z)\bigg|_{z=0}, \ \ \frac{d^2}{dz^2}\Psi_N(z)\bigg|_{z=0}  \overset{N\to\infty}{\longrightarrow}\frac{d^2}{dz^2}\mathsf{E}_\omega(z)\bigg|_{z=0},
\end{align*}
we get that there exist real (finite) $\gamma_1$ and $\delta\ge 0$ such that:
\begin{align*}
  \gamma^{(N)}_1\left(\mathbf{x}^{(N)}\right)\overset{N\to\infty}{\longrightarrow}  \gamma_1, \ \ \delta^{(N)}\left(\mathbf{x}^{(N)}\right)\overset{N\to\infty}{\longrightarrow} \delta
\end{align*}
and then define by $\gamma_2$ by (\ref{Gamma2}).
\end{proof}

\begin{proof}[Proof of Proposition \ref{PropositionExtension}]
Observe that, we can write $\mathsf{\Phi}_N(z)$ as follows:
\begin{equation*}
\mathsf{\Phi}_N(z)=e^{\left(\mathsf{c}^{(N)}-\sum_{i=1}^Nx_i^{(N)}\right)z}\prod_{i=1}^N\left(1-x_i^{(N)}\right)e^{x_i^{(N)}z}.
\end{equation*}
Since by assumption we have:
\begin{equation*}
  \frac{d}{dz}\mathsf{\Phi}_N(z)\bigg|_{z=0}  \overset{N\to\infty}{\longrightarrow}\frac{d}{dz}\mathsf{E}(z)\bigg|_{z=0}<\infty, \ \ \frac{d^2}{dz^2}\mathsf{\Phi}_N(z)\bigg|_{z=0}  \overset{N\to\infty}{\longrightarrow}\frac{d^2}{dz^2}\mathsf{E}(z)\bigg|_{z=0}<\infty,
\end{equation*}
we obtain that there exists a finite constant $K$ such that:
\begin{equation}\label{ConstantBoundedness}
  \sup_{N\ge 1} \left|\mathsf{c}^{(N)}-\sum_{i=1}^Nx_i^{(N)}\right| \le K, \ \ \sup_{N\ge 1} \sum_{i=1}^N \left(x_i^{(N)}\right)^2\le K.
\end{equation}
These sequences converge but we shall not need this. Hence, again using the inequality $\left|\left(1-z\right)e^{z} \right|\le e^{4|z|^2}$, we can bound for all $z\in \mathbb{C}$:
\begin{equation}\label{BoundExtension}
 \left|\mathsf{\Phi}_N(z)\right| \le e^{\left|\mathsf{c}^{(N)}-\sum_{i=1}^Nx_i^{(N)}\right||z|}e^{4\sum_{i=1}^N\left(x_i^{(N)}\right)^2|z|^2} \le e^{K(|z|+4|z|^2)} .
\end{equation}
Thus, the sequence $\left\{\mathsf{\Phi}_N(z)\right\}_{N=1}^\infty$is uniformly bounded on any compact set in $\mathbb{C}$. By Montel's theorem \cite{Ahlfors} it is normal. All entire functions possibly arising as subsequential limits must coincide with $\mathsf{E}$ in a real neighbourhood of the origin and thus by the identity theorem \cite{Ahlfors} they must equal $\mathsf{E}$ on all of $\mathbb{C}$. This prove the convergence statement.

To see that $\mathsf{E}$ belongs in $\mathcal{LP}$ observe the following: $\mathsf{E}(0)=1$, from (\ref{BoundExtension}) $\mathsf{E}$ is of order at most $2$ and finally by Hurwitz's theorem \cite{Ahlfors} its zeros (which are all real) arise as limits of the zeros of $\mathsf{\Phi}_N$ (namely the reciprocals of the $x_i^{(N)}$'s). Then, Hadamard's factorisation theorem \cite{Ahlfors} gives the required form of an $\mathcal{LP}$ function (the sum of squares being finite follows from (\ref{ConstantBoundedness}) and Fatou's lemma).
\end{proof}

We have the following quantitative version of the result stated in Remark \ref{RemarkExtension}, which itself generalises the statement of Proposition \ref{ConvergenceProp}. In plain words, if one has quantitative control on how close $\omega$ is to $\tilde{\omega}$ in $\hat{\Omega}$ then we have uniform control on how close $\mathsf{E}_{\omega}$ is to $\mathsf{E}_{\tilde{\omega}}$. Although we did not make probabilistic use of it, it is plausible that it will have applications in the future. The fact that the bound is compatible with coordinate-wise convergence in $\hat{\Omega}$ (namely it goes to zero as $\omega$ converges to $\tilde{\omega}$ coordinate-wise in $\hat{\Omega}$) is the content of Lemma \ref{LemmaTop} below, see also Remark \ref{CommentsQuantitative} for more comments.

\begin{prop}\label{PropositionQuantitative}
Suppose $\omega=\left(\alpha^+,\alpha^-,\gamma_1,\delta\right), \tilde{\omega}=\left(\tilde{\alpha}^+,\tilde{\alpha}^-,\tilde{\gamma}_1,\tilde{\delta}\right) \in \hat{\Omega}$. Then, there exists some absolute constant $L$ such that, for all $z\in \mathbb{C}$:
\begin{align}\label{QuantitativeBound}
 \left|\mathsf{E}_{\omega}(z)-\mathsf{E}_{\tilde{\omega}}(z)\right|&\le e^{|\gamma_1||z|+5\delta|z|^2}\left(e^{\left|\gamma_1-\tilde{\gamma}_1\right||z|+\frac{|\delta-\tilde{\delta}|}{2}|z|^2}-1\right) \nonumber \\
 &+|z|\left(\sum_{i=1}^\infty\left[\left|\alpha_i^+-\tilde{\alpha}_i^+\right|^3+\left|\alpha_i^- -\tilde{\alpha}_i^-\right|^3\right]\right)^{\frac{1}{3}}e^{|\tilde{\gamma}_1||z|+\frac{\tilde{\delta}}{2}|z|^2+L\left(|z|\left(\delta^{\frac{1}{2}}+\tilde{\delta}^{\frac{1}{2}}\right)+1\right)^{3}}.
\end{align}
\end{prop}

\begin{rmk}
Clearly an analogous bound to (\ref{QuantitativeBound}) holds with the role of $\omega$ and $\tilde{\omega}$ on the right hand side swapped.
\end{rmk}

\begin{lem}\label{LemmaTop}
Assume $\hat{\Omega} \ni\omega_N=(\alpha^{+,(N)},\alpha^{-,(N)},\gamma_1^{(N)},\delta^{(N)}) \overset{N\to \infty}{\longrightarrow}\omega=(\alpha^+,\alpha^-,\gamma_1,\delta)\in \hat{\Omega}$ in the topology of coordinate-wise convergence. Then, we have:
\begin{align*}
 \sum_{i=1}^\infty\left|\alpha_i^{+,(N)}-\alpha_i^+\right|^3\overset{N\to \infty}{\longrightarrow}0, \ \ \sum_{i=1}^\infty\left|\alpha_i^{-,(N)}-\alpha_i^-\right|^3\overset{N\to \infty}{\longrightarrow}0.   
\end{align*}
\end{lem}

\begin{proof}[Proof of Proposition \ref{PropositionQuantitative}]
We can write $\mathsf{E}_\omega(z)=\mathsf{f}_{\omega}(z)\mathsf{g}_\omega(z)$, where:
\begin{align*}
\mathsf{f}_{\omega}(z)&=\prod_{i=1}^\infty\left(1-z\alpha_i^+\right)e^{z\alpha_i^++\frac{z^2}{2}\left(\alpha_i^+\right)^2} \prod_{i=1}^\infty\left(1+z\alpha_i^-\right)e^{-z\alpha_i^-+\frac{z^2}{2}\left(\alpha_i^-\right)^2},\\
\mathsf{g}_{\omega}(z)&=e^{-\gamma_1 z-\frac{\delta}{2}z^2}.
\end{align*}
Using the triangle inequality we get:
\begin{equation*}
  \left|\mathsf{E}_{\omega}(z)-\mathsf{E}_{\tilde{\omega}}(z)\right|  \le \left|\mathsf{f}_\omega(z)\right|\left|\mathsf{g}_\omega(z)-\mathsf{g}_{\tilde{\omega}}(z)\right|+\left|\mathsf{g}_{\tilde{\omega}}(z)\right|\left|\mathsf{f}_\omega(z)-\mathsf{f}_{\tilde{\omega}}(z)\right|.
\end{equation*}
Moreover, using the inequalities $\left|\left(1-z\right)e^{z} \right|\le e^{4|z|^2}$ and $\left|e^{z}-1\right|\le e^{|z|}-1$ we obtain:
\begin{align}
\left|\mathsf{f}_\omega(z)\right|   &\le e^{\frac{9}{2}|z|^2\left(\sum_{i=1}^\infty\left(\alpha_i^+\right)^2+\sum_{i=1}^\infty\left(\alpha_i^-\right)^2\right)} \le e^{\frac{9\delta}{2}|z|^2},\label{Bound1}\\
\left|\mathsf{g}_{\tilde{\omega}}(z)\right| &\le e^{|\tilde{\gamma}_1||z|+\frac{\tilde{\delta}}{2}|z|^2} \label{Bound2} ,\\
\left|\mathsf{g}_{\omega}(z)-\mathsf{g}_{\tilde{\omega}}(z)\right|&\le e^{|\gamma_1||z|+\frac{\delta}{2}|z|^2}\left|e^{\left(\gamma_1-\tilde{\gamma}_1\right)z+\frac{\delta-\tilde{\delta}}{2}z^2}-1\right|\le e^{|\gamma_1||z|+\frac{\delta}{2}|z|^2}\left(e^{\left|\gamma_1-\tilde{\gamma}_1\right||z|+\frac{|\delta-\tilde{\delta}|}{2}|z|^2}-1\right)\label{Bound3}.
\end{align}
Now, we observe that $\mathsf{f}_\omega(z)$ is simply the regularised (of order 3) determinant  $\det_3\left(\mathbf{I}-z\mathfrak{A}\right)$ for an operator $\mathfrak{A}$ with eigenvalues $\alpha_i^+,-\alpha_i^-$, see Theorem 6.2 in \cite{SimonFredholm}. Such an operator that we denote by $\mathfrak{A}_{(\alpha^+,\alpha^-)}$ to emphasise dependence on $\alpha^+,\alpha^-$, can be concretely defined as follows.  On an infinite-dimensional separable Hilbert space $\mathfrak{H}$ with orthonormal basis $\{\mathsf{e}_i^+,\mathsf{e}_i^-\}_{i=1}^\infty$ we define $\mathfrak{A}_{(\alpha^+,\alpha^-)}:\mathfrak{H} \to \mathfrak{H}$ by $\mathfrak{A}_{(\alpha^+,\alpha^-)} \mathsf{e}_i^+=\alpha_i^+\mathsf{e}_i^+$ and $\mathfrak{A}_{(\alpha^+,\alpha^-)} \mathsf{e}_i^-=-\alpha_i^-\mathsf{e}_i^-$ for all $i\ge 1$. Then, we may apply Theorem 6.5 in \cite{SimonFredholm} to obtain:
\begin{align}
 \left|\mathsf{f}_\omega(z)-\mathsf{f}_{\tilde{\omega}}(z)\right|&=\left|\textnormal{det}_3\left(\mathbf{I}-z\mathfrak{A}_{(\alpha^+,\alpha^-)}\right)-\textnormal{det}_3\left(\mathbf{I}-z\mathfrak{A}_{(\tilde{\alpha}^+,\tilde{\alpha}^-)}\right)\right|\nonumber\\ 
 &\le  |z|\left(\sum_{i=1}^\infty\left[\left|\alpha_i^+-\tilde{\alpha}_i^+\right|^3+\left|\alpha_i^- -\tilde{\alpha}_i^-\right|^3\right]\right)^{\frac{1}{3}}\times \nonumber\\
 &\times \exp\left(L\left[|z|\left( \sum_{i=1}^\infty\left[\left(\alpha_i^+\right)^3+\left(\alpha_i^-\right)^3\right]\right)^{\frac{1}{3}}+|z|\left( \sum_{i=1}^\infty\left[\left(\tilde{\alpha}_i^+\right)^3+\left(\tilde{\alpha}_i^-\right)^3\right]\right)^{\frac{1}{3}}+1\right]^3\right) \nonumber\\
 &\le|z|\left(\sum_{i=1}^\infty\left[\left|\alpha_i^+-\tilde{\alpha}_i^+\right|^3+\left|\alpha_i^- -\tilde{\alpha}_i^-\right|^3\right]\right)^{\frac{1}{3}}\exp\left(L\left(|z|\left(\delta^{\frac{1}{2}}+\tilde{\delta}^{\frac{1}{2}}\right)+1\right)^{3}\right), \label{Bound4}
\end{align}
for some constant $L$. In applying Theorem 6.5 of \cite{SimonFredholm} one needs to compute the singular values of the operator $z(\mathfrak{A}_{(\alpha^+,\alpha^-)}-\mathfrak{A}_{(\tilde{\alpha}^+,\tilde{\alpha}^-)})$ (similarly of $z\mathfrak{A}_{(\alpha^+,\alpha^-)}$ and $z\mathfrak{A}_{(\tilde{\alpha}^+,\tilde{\alpha}^-)}$) and we have used the fact that $\mathfrak{A}_{(\alpha^+,\alpha^-)}-\mathfrak{A}_{(\tilde{\alpha}^+,\tilde{\alpha}^-)}$ (similarly  $\mathfrak{A}_{(\alpha^+,\alpha^-)}$ and $\mathfrak{A}_{(\tilde{\alpha}^+,\tilde{\alpha}^-)}$) is self-adjoint. Moreover, the absolute constant $L$ can be picked such that $L\le e\left(2+\log 3\right)$, see the discussion before Theorem 6.4 in \cite{SimonFredholm}. Putting (\ref{Bound1}), (\ref{Bound2}), (\ref{Bound3}) and (\ref{Bound4}) together gives the bound in the statement of the proposition.
\end{proof}

\begin{rmk}\label{CommentsQuantitative}In the proof above we could have used instead the factorisation $\mathsf{E}_\omega(z)=\mathsf{f}_\omega^{\star}(z)\mathsf{g}_\omega^{\star}(z)$ where:
\begin{align*}
   \mathsf{f}_\omega^{\star}(z)&=\prod_{i=1}^\infty\left(1-z\alpha_i^+\right)e^{z\alpha_i^+} \prod_{i=1}^\infty\left(1+z\alpha_i^-\right)e^{-z\alpha_i^-} ,\\
  \mathsf{g}_\omega^{\star}(z)&=e^{-\gamma_1z-\frac{1}{2}\left[\delta-\sum_{i=1}^\infty (\alpha_i^+)^2+\sum_{i=1}^\infty (\alpha_i^-)^2\right]z^2}.
\end{align*}
Then, $\mathsf{f}_\omega^{\star}(z)$ is given by the standard Hilbert-Schmidt regularised determinant $\det_2\left(\mathbf{I}-z\mathfrak{A}_{(\alpha^+,\alpha^-)}\right)$. Moreover, using standard inequalities for reguralised determinants, see Theorem 6.5 in \cite{SimonFredholm}, we obtain an overall better\footnote{In that there is no $e^{C|z|^3}$ factor.} bound than (\ref{Bound4}) but which involves the factor:
\begin{equation*}
   \left(\sum_{i=1}^\infty\left[\left|\alpha_i^+-\tilde{\alpha}_i^+\right|^2+\left|\alpha_i^- -\tilde{\alpha}_i^-\right|^2\right]\right)^{\frac{1}{2}}.
\end{equation*}
Note that, this is no longer compatible with coordinate-wise convergence in $\hat{\Omega}$ as such convergence only implies weak convergence in the sequence space $\ell^2$ for the $\alpha^+$ and $\alpha^-$ parameters (viewed as elements of $\ell^2$). On the other hand, as we prove in Lemma \ref{LemmaTop}, it implies norm convergence in $\ell^3$.
\end{rmk}

\begin{rmk}It is also possible to prove a higher order result of this type. For any $r\ge 2$, we consider the infinite-dimensional space $\hat{\Omega}^{(r)}$ (observe that $\hat{\Omega}^{(2)}\equiv \hat{\Omega}$):
\begin{align*}
\hat{\Omega}^{(r)}&= \bigg\{\omega=\left(\alpha^+,\alpha^-,\delta_1,\delta_2,\dots,\delta_r\right)\in \mathbb{R}_+^\infty\times \mathbb{R}_+^\infty \times \mathbb{R}\times \mathbb{R} \times \cdots \times \mathbb{R}_+:\\
&\alpha^+=(\alpha_1^+\ge \alpha_2^+\ge \cdots \ge 0);\alpha^-=(\alpha_1^-\ge \alpha_2^-\ge \cdots \ge 0);\sum_{i=1}^\infty\left(\alpha_i^+\right)^r +\sum_{i=1}^\infty\left(\alpha_i^-\right)^r\le \delta_r \bigg\}.
\end{align*}
Endow $\hat{\Omega}^{(r)}$ with the topology of coordinate-wise convergence. Furthermore, consider the class of functions, parametrised by $\omega\in \hat{\Omega}^{(r)}$, and given by:
\begin{equation*}
 \mathfrak{E}_\omega(z)= e^{-\sum_{j=1}^r \frac{\delta_jz^j}{j}}   \prod_{i=1}^\infty\left(1-z\alpha_i^+\right)e^{\sum_{j=1}^r\frac{(\alpha_i^+)^jz^j}{j}} \prod_{i=1}^\infty\left(1+z\alpha_i^-\right)e^{\sum_{j=1}^r\frac{(-\alpha_i^+)^jz^j}{j}}.
\end{equation*}
Write $\mathfrak{E}_\omega(z)=\mathfrak{g}_\omega(z)\mathfrak{f}_\omega(z)$, where:
\begin{equation*}
  \mathfrak{f}_\omega(z) = \prod_{i=1}^\infty\left(1-z\alpha_i^+\right)e^{\sum_{j=1}^r\frac{(\alpha_i^+)^jz^j}{j}} \prod_{i=1}^\infty\left(1+z\alpha_i^-\right)e^{\sum_{j=1}^r\frac{(-\alpha_i^+)^jz^j}{j}}, \ \ \mathfrak{g}_\omega(z)= e^{-\sum_{j=1}^r \frac{\delta_jz^j}{j}}  .
\end{equation*}
We observe that, $\mathfrak{f}_\omega(z)$ is also a regularised determinant: $\mathfrak{f}_\omega(z)=\det_{r+1}\left(\mathbf{I}-z\mathfrak{A}_{(\alpha^+,\alpha^-)}\right)$, see Theorem 6.2 in \cite{SimonFredholm}. Then, by adapting the proof of Proposition \ref{PropositionQuantitative}, using the standard inequality:
\begin{equation*}
 \left|(1-z)e^{\sum_{j=1}^{r-1} \frac{z^j}{j}}\right|\le e^{C_r |z|^r}, \ \ \forall z\in \mathbb{C}, 
\end{equation*}
for some universal constant $C_r$ and the standard inequalities for the regularised determinant $\det_{r+1}$, see Theorem 6.5 in \cite{SimonFredholm}, we obtain an analogous uniform bound to (\ref{QuantitativeBound}) for $\left|\mathfrak{E}_\omega(z)-\mathfrak{E}_{\tilde{\omega}}(z)\right|$. Finally, Lemma \ref{LemmaTop} also has a suitable generalisation. At present, we do not have any probabilistic applications in mind for this more general setup but it might be useful in the future.
\end{rmk}

\begin{proof}[Proof of Lemma \ref{LemmaTop}]
We only prove the result for the $\alpha^+$ parameters. The proof for the $\alpha^-$ parameters is completely analogous. We assume, and we will prove shortly,  that:
\begin{equation}\label{NormConvergence}
 \sum_{i=1}^\infty\left(\alpha_i^{+,(N)}\right)^3\overset{N\to \infty}{\longrightarrow} \sum_{i=1}^\infty\left(\alpha_i^{+}\right)^3.   
\end{equation}
Moreover, we observe: 
\begin{equation*}
 2^3\left[\left(\alpha_i^{+,(N)}\right)^3+\left(\alpha_i^{+}\right)^3\right]- \left|\alpha_i^{+,(N)}-\alpha_i^+\right|^3 \ge 0, \ \ \forall i\ge 1, N\ge 1.
\end{equation*}
Then, using Fatou's lemma and the claim (\ref{NormConvergence}) above we obtain the desired conclusion. 

Now, in order to prove the claim we adapt the proof of Proposition 2.3 in  \cite{OrbitalBeta}. Let $r\ge 1$ be arbitrary. We can write:
\begin{align*}
  \sum_{i=1}^\infty\left(\alpha_i^{+,(N)}\right)^3&=\sum_{i=1}^r\left(\alpha_i^{+,(N)}\right)^3+\mathcal{O}\left(\alpha_{r+1}^{+,(N)}\sum_{j=1}^\infty \left(\alpha_j^{+,(N)}\right)^2\right).
\end{align*}
The upper and lower limits of this display, as $N\to \infty$, can be written as:
\begin{equation*}
      \sum_{i=1}^r\left(\alpha_i^{+}\right)^3 +\mathcal{O}\left(\sup_{N\ge 1}\delta^{(N)}\alpha_{r+1}^+\right) .
\end{equation*}
On the other hand, since:
\begin{equation*}
      \sum_{i=r+1}^\infty\left(\alpha_i^{+}\right)^3=\mathcal{O}\left(\alpha_{r+1}^+\sum_{i=1}^\infty \left(\alpha_i^+\right)^2\right),
\end{equation*}
we deduce that the upper and lower limits of $ \sum_{i=1}^\infty\left(\alpha_i^{+,(N)}\right)^3 $ are both equal to:
\begin{equation*}
   \sum_{i=1}^\infty\left(\alpha_i^{+}\right)^3 +\mathcal{O}\left(\sup_{N\ge 1}\delta^{(N)}\alpha_{r+1}^+\right) .
\end{equation*}
Letting $r\to \infty$ gives (\ref{NormConvergence}).
\end{proof}

Write $\mathfrak{T}:\Omega \to \mathcal{LP}$ for the map $\omega \mapsto \mathsf{E}_\omega(z)$ induced by Definition \ref{DefLP}. We have the following result.

\begin{prop}\label{MeasureProp}
The pushforward map:
\begin{align*}
  \mathfrak{T}_*:\mathcal{M}_p\left(\Omega\right) &\to \mathcal{M}_p\left(\mathcal{LP}\right)  \\
 \nu(\cdot) &\mapsto \nu\left(\mathfrak{T}^{-1}\left(\cdot\right)\right),
\end{align*}
is well-defined and is in fact a bijection.
\end{prop}

\begin{proof}
We prove that $\mathfrak{T}$ is a homeomorphism. First observe that it is a bijection. Moreover, since by definition the bijection between $\hat{\Omega}$ and $\Omega$ is bi-continuous it suffices to establish the claim for the induced map $\hat{\mathfrak{T}}:\hat{\Omega} \to \mathcal{LP}$. Note that, from Proposition \ref{PropositionQuantitative} and Lemma \ref{LemmaTop} the map $\hat{\mathfrak{T}}$ is continuous\footnote{Both $\hat{\Omega}$ and $\mathcal{LP}$ are first-countable so continuity is equivalent to sequential continuity.} (in fact we do not need the quantitative statement, the result stated in Remark \ref{RemarkExtension} suffices). On the other hand, $\hat{\mathfrak{T}}^{-1}$ is also continuous which follows from word for word adaptation of the proof of the if direction of Proposition \ref{ConvergenceProp}. This proves the claim. Then, the statement of the proposition is a consequence of the following standard fact: suppose $\tau:\mathfrak{X} \to \mathfrak{Y}$ is a homeomorphism. Then, the pushforward map $\tau_*:\mathcal{M}_p\left(\mathfrak{X}\right) \to \mathcal{M}_p\left(\mathfrak{Y}\right)$ is a bijection.
\end{proof}

\begin{prop}\label{PrincipalValProp}
The following principal value product representation holds
\begin{equation}\label{PrincValProd}
    \mathsf{E}_\omega(z)= \lim_{R \to \infty} \prod_{\left|\alpha_i^{\pm}\right|>R^{-2}}\left(1-\alpha_i^+z\right)\left(1+\alpha_i^-z\right)
\end{equation}
if and only if
\begin{equation}\label{PrincValCond}
\gamma_1= \lim_{R \to \infty} \sum_{\left|\alpha_i^{\pm}\right|>R^{-2}}\left(\alpha_i^+-\alpha_i^-\right)  \textnormal{ and }  \gamma_2=0.
\end{equation}
\end{prop}

\begin{proof}
Assuming (\ref{PrincValCond}) we can manipulate the standard form of $\mathsf{E}_\omega$ as follows:
\begin{align*}
\mathsf{E}_\omega(z)&=e^{-\lim_{R\to \infty}\sum_{\left|\alpha_i^{\pm}\right|>R^{-2}}\left(\alpha_i^+-\alpha_i^-\right)z}\prod_{i=1}^\infty \left(1-\alpha_i^+z\right)e^{\alpha_i^+z}\left(1+\alpha_i^-z\right)e^{-\alpha_i^-z}\\
&=\lim_{R \to \infty}e^{-\sum_{\left|\alpha_i^{\pm}\right|>R^{-2}}\left(\alpha_i^+-\alpha_i^-\right)z}\prod_{i=1}^\infty \left(1-\alpha_i^+z\right)e^{\alpha_i^+z}\left(1+\alpha_i^-z\right)e^{-\alpha_i^-z}\\
&=\lim_{R \to \infty} \prod_{\left|\alpha_i^{\pm}\right|>R^{-2}}e^{-\alpha_i^+z+\alpha_i^-z} \lim_{R \to \infty}  \prod_{\left|\alpha_i^{\pm}\right|>R^{-2}}\left(1-\alpha_i^+z\right)e^{\alpha_i^+z}\left(1+\alpha_i^-z\right)e^{-\alpha_i^-z}\\
&=\lim_{R \to \infty} \prod_{\left|\alpha_i^{\pm}\right|>R^{-2}}\left(1-\alpha_i^+z\right)\left(1+\alpha_i^-z\right).
\end{align*}
Conversely assume (\ref{PrincValProd}). Then, for any $z\in \mathbb{C}\backslash \left\{\frac{1}{\alpha_i^+}\right\}\sqcup\left\{-\frac{1}{\alpha_i^-}\right\}$ we have
\begin{align*}
\lim_{R \to \infty} e^{-\sum_{\left|\alpha_i^{\pm}\right|>R^{-2}}\left(\alpha_i^+-\alpha_i^-\right)z}= \frac{\lim_{R \to \infty} \prod_{\left|\alpha_i^{\pm}\right|>R^{-2}}\left(1-\alpha_i^+z\right)\left(1+\alpha_i^-z\right)}{\lim_{R \to \infty} \prod_{\left|\alpha_i^{\pm}\right|>R^{-2}}\left(1-\alpha_i^+z\right)e^{\alpha_i^+z}\left(1+\alpha_i^-z\right)e^{-\alpha_i^-z}}= e^{-\gamma_1 z-\frac{\gamma_2}{2} z^2},
\end{align*}
from which (\ref{PrincValCond}) follows.
\end{proof}

\bibliographystyle{acm}
\bibliography{References}

\bigskip 

\noindent{\sc School of Mathematics, University of Edinburgh, James Clerk Maxwell Building, Peter Guthrie Tait Rd, Edinburgh EH9 3FD, U.K.}\newline
\href{mailto:theo.assiotis@ed.ac.uk}{\small theo.assiotis@ed.ac.uk}

\end{document}